\documentclass[12pt, english]{amsart}
\usepackage[english]{babel}
\usepackage[utf8]{inputenc}

\usepackage{graphicx,subfigure}
\usepackage{amstext}
\usepackage{amssymb}
\usepackage{amsmath}
\usepackage{amsthm}
\usepackage{stmaryrd}
\usepackage{fullpage}
\usepackage[all]{xy}
\xyoption{arc}
\CompileMatrices
\usepackage{MnSymbol}
\usepackage{dynkin-diagrams}

\usepackage[colorlinks=true,citecolor=blue]{hyperref}

\usepackage{tikz-cd}
\usetikzlibrary{backgrounds,fit, matrix}
\usetikzlibrary{positioning}
\usetikzlibrary{calc,through,chains}
\usetikzlibrary{arrows,shapes,decorations.pathmorphing,automata, petri}
\usepackage[boxsize=1.25em, centerboxes]{ytableau}

\usepackage[
    backend=biber,
    style=alphabetic,
  ]{biblatex}
\newtheorem{Theorem}{Theorem}
\newtheorem{Corollary}[Theorem]{Corollary}
\newtheorem{Lemma}[Theorem]{Lemma}
\newtheorem{Proposition}[Theorem]{Proposition}

\newtheorem{Exercise}[Theorem]{Exercise}

\theoremstyle{definition}
\newtheorem{Definition}[Theorem]{Definition}
\newtheorem{Example}[Theorem]{Example}

\newtheorem{Remark}[Theorem]{Remark}

\numberwithin{equation}{section}
\numberwithin{figure}{section}

\DeclareMathOperator{\Spec}{Spec}

\DeclareMathOperator{\Gr}{Gr}
\DeclareMathOperator{\Hom}{Hom}

\DeclareMathOperator{\Pic}{Pic}

\DeclareRobustCommand{\GL}{\mathrm{GL}}
\DeclareRobustCommand{\SL}{\mathrm{SL}}
\DeclareRobustCommand{\PGL}{\mathrm{PGL}}
\DeclareRobustCommand{\SO}{\mathrm{SO}}
\DeclareRobustCommand{\Sp}{\mathrm{Sp}}

\DeclareMathOperator{\PExp}{PExp}

\DeclareMathOperator{\PLog}{PLog}

\DeclareMathOperator{\Aut}{Aut}
\DeclareMathOperator{\End}{End}
\DeclareMathOperator{\tr}{tr}
\DeclareMathOperator{\Spin}{Spin}

\makeatletter
\@namedef{subjclassname@2020}{%
  \textup{2020} MSC}
\makeatother

\title{$e$-polynomials of character varieties}

\author{Alfonso Zamora}

\address{Departamento de Matemática Aplicada a las TIC, ETSI Informáticos, Universidad Politécnica de Madrid, Campus de Montegancedo, 28660 Madrid, Spain}
  \email{alfonso.zamora@upm.es}
\subjclass[2020]{Primary: 14M35. Secondary: 14L24, 14D20, 14C15, 16G99, 17B22, 32S35}
\keywords{Character varieties, e-polynomials, mixed hodge structures, topological mirror symmetry, Langlands duality, representation theory}

\makeatother

\begin{document}

\begin{abstract}
This manuscript contains the basic ideas and constructions about $e$-polynomials in character varieties and the state of the art of certain research in the field, plus some new further directions. 
We introduce mixed Hodge structures and $e$-polynomials, together with a series of arithmetic (counting points over finite fields) and geometric (stratification into parabolic types) techniques to compute them. We include a complete example of the calculation of the $e$-polynomial for the $\GL_3$-character variety of the free group. Finally, we extend the geometric stratification into parabolic types to a general reductive group $G$ to obtain explicit motivic expressions for the $G$-character varieties, and reduce certain topological mirror symmetry conjectures for these moduli spaces. 
\end{abstract}

\maketitle

\tableofcontents

\section*{Introduction}
%\addcontentsline{toc}{section}{Introduction}

Given a complex reductive algebraic group $G$, and  a finitely presented group $\Gamma$, the $G$-character variety is defined as the GIT quotient $\Hom(\Gamma, G)/\!\!/ G$. When $\Gamma$ is the fundamental group of a Riemann surface, this space has been studied quite well because it is homeomorphic to the space of $G$-Higgs bundles via the non-abelian Hodge correspondence \cite{simpson}. This spaces also play an important role in mathematical physics in the context of mirror symmetry \cite{Kapustin2006}: for the moduli spaces of the pair of Langlands dual groups $\SL_n$ and $\PGL_n$ (which are smooth/orbifold) it is conjectured that certain topological and geometrical invariants coincide or mirror in a certain way. Other groups $\Gamma$ provide very singular character varieties for which the topological and geometrical invariants are much more unexplored and difficult to compute.

Among the interesting invariants we can be interested in we first find Poincaré polynomials encoding Betti numbers, whose computation for surface groups started with Hitchin \cite{hitchin1987} and Gothen \cite{gothen1994} with recent major advances by Schiffmann \cite{schiffmann2016} or Mellit \cite{Mellit2020}. After Deligne's work \cite{deligne1971theorie, deligne1974theorie} on mixed Hodge structures for singular spaces, the computation of mixed Hodge numbers, and the polynomials encoding these such as the mixed Hodge polynomial or the $e$-polynomial, became interesting as a way to capture intermediate information between topology and complex geometry. 
This was initiated by Hausel and Rodriguez-Villegas in the surface case \cite{hausel2008mixed} by using arithmetic methods due to the fact that the number of points of the moduli space over finite fields is given by a polynomial which, in turn, coincides with the $e$-polynomial \cite[Katz's Appendix]{hausel2008mixed}. This line of thinking has been continued by other authors as in \cite{baraglia2017arithmetic} and has developed an important interaction between arithmetics and representation theory by counting indecomposable or irreducible representations to produce formulae yielding these polynomials \cite{reineke2003harder, reineke2006counting, mozgovoy2007computational, mozgovoy2009number, mozgovoy2015arithmetic}. One major achievement was to prove a topological mirror symmetry statement for $\SL_n$ and $\PGL_n$, showing that their $e$-polynomials are equal (up to passing to their stringy version accounting for orbifold de-singularizations) \cite{Hausel_Thaddeus_2003, hausel2008mixed, groechenig2020}. 

The seminal paper \cite{logares2013hodge} uses a geometric decomposition of the character variety into subvarieties of different stabilizers. The good properties of the $e$-polynomial under locally closed stratifications and certain fibrations allow to produce a polynomial for the whole moduli space by summing up the polynomials of the strata which, in turn, are computed by means of fibrations with the stabilizers as kernels. 
This has been shown to be successful for small rank character varieties for $\GL_n$, $\SL_n$ and $\PGL_n$ as in \cite{FlorentinoLawton09,  FlorentinoLawton14, lawton2016SL3, CasimiroFlorentinoLawtonOliveira16, martinez2016thesis}. For ranks higher or equal than $4$ computations seem to be intractable yet. 

Later, the papers \cite{florentino2021serre, florentino2023generating} combine the arithmetic and the geometric techniques in the groups $\GL_n$, $\SL_n$ and $\PGL_n$ to provide generating functions for the polynomial expressions based in plethystic expressions coming from high energy physics \cite{Feng2007}, combined with a good understanding of the representation theory of symmetric groups \cite{cheah1996cohomology}. These produce a stratification of the character variety into polystability types, where the big stratum corresponds to the irreducible representations and the smallest one are the abelian representations studied in \cite{florentino2021abelian, li2024SPSO}. When $\Gamma$ is the free group, the character varieties we obtain are very singular.  In \cite{florentino2021serre} actual explicit expressions for these character varieties are derived, based on those from \cite{mozgovoy2015arithmetic}, and the authors prove a $\SL_n-\PGL_n$ topological mirror symmetry statement. The paper \cite{gonzalez2024rootdata} explores a generalization of this to other groups $G$, producing a stratification into parabolic subgroups of $G$ and reducing topological mirror symmetry to checking it strata by strata. This uses the idea of core and pseudo quotient in \cite{gonzalez2024pseudo} which generalizes the notion of polystable representatives in a moduli space to each particular parabolic strata. 

The outline of the notes is the following. In Section 1 we cover the basics of Deligne's mixed Hodge structures and in Section 2 we define the invariants we want to compute, the $e$-polynomials, and explore its properties. Section 3 is devoted to an overview on the main techniques to compute $e$-polynomials. Section 4 contains a complete example of the computation for the $\GL_3$-character variety of the free group, using most of the ingredients previously introduced. Up to here, these notes contain several exercises towards a better acquisition of the concepts of the first 4 sections. Section 5 explores the generalization for complex reductive groups $G$ and section 6 contains actual computations and topological mirror symmetry statements and conjectures for simple groups of Dynkin type ABCD. 

\subsection*{Acknowledgements}
I would like to thank Ronald. A. Zúñiga-Rojas and Alexander Schmitt for inviting me to deliver a series of talks at the Workshop on Character Varieties and Higgs Bundles held in Liberia, Guanacaste, Costa Rica, in August 2025. I thank the University of Costa Rica sede Guanacaste for hospitality while delivering the course and writing this manuscript. I also thank Marwan Benyoussef, Ana Cristina Casimiro, Guillermo Gallego, Cesare Goretti, Shengshuan Liu, Allan Murillo, and the rest of students and participants of the workshop for their comments and suggestions on correcting and improving this work. Finally, I thank the referee for the attentive reading of the manuscript and for his/her comments. This work was partially supported by project PID2022-142024NB-I00 by the Spanish government.

\section{Mixed Hodge structures}

In this section we present the notion of mixed Hodge structure, introduced by Deligne in \cite{deligne1971theorie, deligne1974theorie} as an extension of classical Hodge theory of compact K\"ahler varieties to the non-compact and/or non-smooth case. For a complete treatment, the reader can consult the book \cite{peters2008ergebnisse}.

Let us begin by reviewing basic notions of Hodge theory.

\begin{Definition}
A \textbf{(pure) Hodge structure (of weight $k$)} on a $\mathbb{Z}$-module $V_{\mathbb{Z}}$ is the following decomposition of its complexification:
\[V:=V_{\mathbb{C}}=V_{\mathbb{Z}}\otimes_{\mathbb{Z}}\mathbb{C} = \bigoplus_{p+q=k} V^{p,q}, \;\; V^{p,q}=\overline{V^{q,p}}\; .\]
Equivalently, we can define what we call a \textbf{Hodge filtration} which is an index decreasing filtration
\[V\supset \cdots \supset F^p(V)\supset F^{p+1}(V)\supset \cdots\; ,\]
where $F^p(V)\cap \overline{F^q(V)}=V^{p,q}$ and $F^p(V)=\bigoplus_{i\geq p}V^{i, k-i}$. 
\end{Definition}

De Rham cohomology modules of compact K\"ahler manifolds naturally have a pure Hodge structure of the weight given by the order of the cohomology. The pieces of the Hodge structure are given by the Dolbeault cohomology. 
 
\begin{Theorem}[\textbf{Hodge decomposition}]
Let $X$ be a compact K\"ahler manifold,  $k^{th}$-cohomology carries a (pure) Hodge structure of weight $k$:
\[H_{DR}^{k}(X)\otimes \mathbb{C}=\bigoplus_{p+q=k}H^{p,q}(X)\; ,\]
where each graded piece is $H^{p,q}(X)\simeq H^q(X, \Omega^p)$, the $q^{th}$-cohomology of the $p$-differential, by Dolbeault's theorem. It satisfies $H^{p,q}(X)=\overline{H^{q,p}(X)}$, and 
the dimensions $h^{p,q}=\dim_{\mathbb{C}} H^{p,q}(X)$ are called the \textbf{Hodge numbers}.    
\end{Theorem}

\begin{Remark}
   In particular, smooth projective varieties are K\"ahler manifolds and Hodge decomposition holds for these.  
\end{Remark}

A nice pictorial way to understand Hodge numbers is through the \textbf{Hodge diamond}:

\begin{equation}\label{pic:diamond}
\begin{array}{ccccccc}
& &  &   h^{n,n}  &  & & \\ 
& &  h^{n,n-1} &  & h^{n-1,n}  & &\\ 
& \udots &  &  &  & \ddots& \\
h^{n,0} & h^{n-1,1}  & & \cdots & &   h^{1,n-1} & h^{0,n}\\ 
& \ddots &  &  &  &\udots & \\
& &  h^{1,0} &  & h^{0,1} & &\\ 
& &    & h^{0,0} & &  &  \\
\end{array} 
\end{equation}

The Hodge diamond satisfies certain symmetries:
\begin{itemize}
    \item \textbf{Serre duality} $H^{p,q}(X)\simeq H^{n-p,n-q}(X)^{\vee}$ yields $h^{p,q}=h^{n-p, n-q}$. Then, the diamond is symmetric with respect to its center. 
    \item \textbf{Hodge symmetry} $H^{p,q}(X)=\overline{H^{q,p}(X)}$ yields $h^{p,q}=h^{q,p}$. Then, the diamond is symmetric with respect to the vertical axis. By composing with the central symmetry, it is also symmetric with respect to the horizontal axis. 
\end{itemize}

Hodge numbers come from De Rham cohomology, therefore they have to encode, in particular, the topology of the variety. By summing up the number in each row of the Hodge diamond we obtain the \textbf{Betti numbers} of $X$:
\[b_k=\dim H^k(X,\mathbb{C})=\sum_{p+q=k} h^{p,q}.\]

Let us show explicitly two main examples of pure Hodge structures: Riemann surfaces and complex projective spaces. 

\begin{Example}\label{ex:diamond_RS}
Let $X$ be a smooth complex projective genus $g$ curve or, equivalently, a compact Riemann surface. Its Hodge diamond is given by:
\[\begin{array}{ccc}
 & h^{1,1}=1 & \\
 h^{1,0}=g=\dim H^0(X, \Omega) &  & h^{0,1}=g\\
 & h^{0,0}=1 &
\end{array} \]
with Betti numbers $b_0=1$, $b_1=2g$, $b_2=1$. 
\end{Example}

\begin{Example}\label{ex:diamond_PS}
Let $\mathbb{P}_{\mathbb{C}}^n$ be the complex projective space of dimension $n$. Its Hodge diamond is given by 
\[\begin{array}{ccccccc}
& &  &   h^{n,n}=1 &  & & \\ 
& &  h^{n-1,1}=0 &  & h^{n,n-1}=0  & &\\ 
& \udots &   & h^{n-1,n-1}=1 &   & \ddots &\\ 
h^{i,0}=0 &   & & \vdots & &    & h^{0,i}=0\\ 
& \ddots&   & h^{1,1}=1 &   & \udots &\\ 
&  &  h^{1,0}=0 &  & h^{0,1}=0 & &\\ 
& &    & h^{0,0}=1 & &  &  \\
\end{array} \]
with Betti numbers $b_{2k}=1$, $k=0,1,\ldots, n$. 
\end{Example}

When we deal with complex varieties which are not smooth nor projective, classical Hodge decomposition fails. Deligne \cite{deligne1971theorie, deligne1974theorie} generalized this by extending the notion to Mixed Hodge Structures. 

\begin{Definition}
A \textbf{mixed Hodge structure} on $V:=V_{\mathbb{C}}=V_{\mathbb{Z}}\otimes_{\mathbb{Z}}\mathbb{C}$ is given by a decreasing Hodge filtration $F^{\bullet}(V)$
\[V\supset \cdots \supset F^p(V)\supset F^{p+1}(V)\supset \cdots\; , \]
together with an increasing \textbf{weight filtration} $W_{\bullet}$
\[0\subset \cdots W_{k-1}\subset W_k\subset \cdots \subset V\; ,\]
such that $F^{\bullet}(V)$ induces a (pure) weight $k$ Hodge structure on graded pieces $Gr_k^{W}(V):=W_k / W_{k-1}$. We define $V^{p,q}:= Gr_F^p Gr_{p+q}^W(V)$
and their dimensions $\dim_{\mathbb{C}} V^{p,q}(X)$ are called the \textbf{mixed Hodge numbers}.
\end{Definition}

Recall that \textbf{cohomology with compact support} is defined as the direct limit
\[H_c^{\bullet}(X)=\underset{\underset{K\subset X \, \text{compact}}{\longrightarrow}}{\lim} H^{\bullet}(X, X\backslash K)\]
taken over cohomology relative to the complement of a compact subset. This is the cohomology given by singular co-chains with support in some compact subset of $X$. See \cite{hatcher2002algebraic} for details. 

\begin{Theorem}\cite[Théorème 3.2.5]{deligne1971theorie}\cite[Proposition 8.3.9]{deligne1974theorie}
If $X$ is a quasi-projective algebraic variety (not necessarily smooth nor complete nor irreducible), singular cohomology $H^k(X)$ and singular compactly supported cohomology $H^k_c(X)$ carry \textbf{mixed Hodge structures}.
\end{Theorem}

Therefore, for $X$ a quasi-projective variety and each $k$, the complex vector space $H^k(X)$ is endowed with a mixed Hodge structure yielding two filtrations whose associated graded is $V^{k;p,q}:= Gr_F^p Gr_{p+q}^W(H^k(X))$. We will call \textbf{mixed Hodge numbers} to the dimensions of these complex vector spaces, $h^{k,p,q}:=\dim_{\mathbb{C}} V^{k;p,q}$. They satisfy $h^{k,p,q}=h^{k,q,p}$ and note that it can be $p+q\neq k$. We call the 
\textbf{$k$-weights} of the mixed Hodge structure to the pairs $(p,q)$ such that $h^{k,p,q}\neq 0$. 

Similarly, we will denote by $h_c^{k,p,q}$ the \textbf{(compactly supported) mixed Hodge numbers} of the singular cohomology with compact support $H^k_c(X)$.
Observe that mixed Hodge structures yield compactly supported Betti numbers by $\dim H_{c}^{k}(X)= \sum_{p,q} h_c^{k,p,q}$. In the smooth case, by Poincaré duality, these give the usual Betti numbers. 

\begin{Remark}\label{rem:weights}
The existence of weights different from the order $k$ of the cohomology modules measure the failure of $X$ from being smooth or projective. In fact, if $X$ is non-singular then weights are $\geq k$ (with $p,q\leq k$) while if 
$X$ is projective, weights are $\leq k$ (\cite[Théorème 8.2.4]{deligne1974theorie}, \cite[Proposition 4.20]{peters2008ergebnisse}). Weight filtration stratifies this failure by codimension. 
\end{Remark}

We end this section with a toy but a key example from \cite{heinloth2024pisa}, which shows that diffeomorphic complex varieties (with the same differential structure) can have different Hodge structures. Moreover, these different Hodge structures can be pure and mixed. Then, mixed Hodge structures are invariants of the complex structure but they do not recognize the smooth structure of the variety. 

\begin{Example}\label{ex:2structures}

Let $\Sigma_1$ be a complex elliptic curve. Topologically, $\Sigma_1$ is homeomorphic to a torus $S^1\times S^1$. As a complex variety, $\Sigma_1$ is analytically isomorphic to the quotient of the complex plane by a lattice, $\mathbb{C}/(\mathbb{Z} +  \tau \mathbb{Z})$, where 
$\tau\in\mathbb{C}$. 

The fact that $\Sigma_1$ is a complex variety with a group structure makes its cotangent bundle $T^{\ast}\Sigma_1$ a trivial bundle, then $T^{\ast}\Sigma_1\simeq \mathbb{C}/(\mathbb{Z} +  \tau \mathbb{Z})\times \mathbb{C}$. Therefore, we can construct diffeomorphisms
\[T^{\ast}\Sigma_1\simeq \mathbb{C}/(\mathbb{Z} +  \tau \mathbb{Z})\times \mathbb{C}\simeq S^1\times S^1\times \mathbb{R}^2\simeq \mathbb{C}^{\ast}\times\mathbb{C}^{\ast}\]
between the cotangent bundle $T^{\ast}\Sigma_1\simeq \Sigma_1\times \mathbb{C}$ and the affine complex algebraic variety $\mathbb{C}^{\ast}\times\mathbb{C}^{\ast}$. Let us compute the cohomology of these two algebraic structures. 

Since $\mathbb{C}$ is a contractible topological space, the singular cohomology of the trivial cotangent bundle $T^{\ast}\Sigma_1\simeq \Sigma_1\times \mathbb{C}$ equals the cohomology of the elliptic curve $\Sigma_1$ itself, which is a genus one Riemann surface. Therefore, its Hodge structure yields a pure Hodge diamond, as in Example \ref{ex:diamond_RS}:
\[\begin{array}{ccc}
 & h^{1,1}=1 & \\
 h^{1,0}=g=1 &  & h^{0,1}=g=1\\
 & h^{0,0}=1 &
\end{array} \]
Observe that the weights $0,1,2$ coincide with the degree of the cohomology in each row, i.e. $p+q=k$, i.e. the Hodge structure is pure. 

On the other hand, the cohomology of $\mathbb{C}^{\ast}\times\mathbb{C}^{\ast}$ is, by K\"unneth isomorphism, 
\[H^{\ast}(\mathbb{C}^{\ast}\times\mathbb{C}^{\ast})\simeq H^{\ast}(\mathbb{C}^{\ast})\otimes H^{\ast}(\mathbb{C}^{\ast}).\] 
Because $\mathbb{C}^{\ast}$ is homotopically equivalent to $S^1$, its cohomology is \[H^{\bullet}(\mathbb{C}^{\ast})=H^0(\mathbb{C}^{\ast})\oplus H^1(\mathbb{C}^{\ast})=\mathbb{C}\oplus\mathbb{C}.\]
The only weights which can appear in the degree $1$ cohomology $H^1(\mathbb{C}^{\ast})$ (of a smooth non-projective complex variety, see Remark \ref{rem:weights}) are $h^{1,1,0}=\dim H^{1,0}(\mathbb{C}^{\ast})$, $h^{1,0,1}=\dim H^{0,1}(\mathbb{C}^{\ast})$ and $h^{1,1,1}=\dim H^{1,1}(\mathbb{C}^{\ast})$. Given that, necessarily, by duality of the mixed Hodge numbers, we have the isomorphism
\[H^{1,0}(\mathbb{C}^{\ast})\simeq H^{0,1}(\mathbb{C}^{\ast})\; ,\] 
and the following sum has to be satisfied
\[h^{1,1,0}+h^{1,0,1}+h^{1,1,1}=\dim H^{1,0}(\mathbb{C}^{\ast})+\dim H^{0,1}(\mathbb{C}^{\ast})+\dim H^{1,1}(\mathbb{C}^{\ast})=\dim H^{1}(\mathbb{C}^{\ast})=1\; ,\] 
we conclude that $H^{1,0}(\mathbb{C}^{\ast})=H^{0,1}(\mathbb{C}^{\ast})=0$ and $\dim H^{1,1}(\mathbb{C}^{\ast})\simeq \mathbb{C}$.
Therefore
\[H^{\bullet}(\mathbb{C}^{\ast})=H^{0,0}(\mathbb{C}^{\ast})\oplus H^{1,1}(\mathbb{C}^{\ast}).\] 
Now we use the K\"unneth isomorphism in cohomology to have
\[H^0\left(\mathbb{C}^{\ast}\times \mathbb{C}^{\ast}\right)\simeq H^{0}(\mathbb{C}^{\ast})\otimes H^{0}(\mathbb{C}^{\ast})=H^{0,0}(\mathbb{C}^{\ast})\otimes H^{0,0}(\mathbb{C}^{\ast}) =\mathbb{C}\otimes \mathbb{C}=\mathbb{C}=H^{0,0}\left((\mathbb{C}^{\ast})^2\right)\]
\[H^1\left(\mathbb{C}^{\ast}\times \mathbb{C}^{\ast}\right)\simeq \left(H^{1}(\mathbb{C}^{\ast})\otimes H^{0}(\mathbb{C}^{\ast})\right)\oplus \left(H^{0}(\mathbb{C}^{\ast})\otimes H^{1}(\mathbb{C}^{\ast})\right)=
\]
\[\left(H^{1,1}(\mathbb{C}^{\ast})\otimes H^{0,0}(\mathbb{C}^{\ast})\right)\oplus \left(H^{0,0}(\mathbb{C}^{\ast})\otimes H^{1,1}(\mathbb{C}^{\ast})\right)=\left(\mathbb{C}\otimes\mathbb{C}\right)\oplus \left(\mathbb{C}\otimes\mathbb{C}\right)=\mathbb{C}\oplus\mathbb{C}=\mathbb{C}^2=H^{1,1}\left((\mathbb{C}^{\ast})^2\right)\]
\[H^2\left(\mathbb{C}^{\ast}\times \mathbb{C}^{\ast}\right)\simeq H^{1}(\mathbb{C}^{\ast})\otimes H^{1}(\mathbb{C}^{\ast})=H^{1,1}(\mathbb{C}^{\ast})\otimes H^{1,1}(\mathbb{C}^{\ast}) =\mathbb{C}\otimes \mathbb{C}=\mathbb{C}=H^{2,2}\left((\mathbb{C}^{\ast})^2\right)\]
Hence, its mixed Hodge structure has Hodge numbers given by 
\begin{eqnarray*}
    h^{0,0,0}\left((\mathbb{C}^{\ast})^2\right)=1, & \text{weight 0 in the }0^{th}\text{-cohomology,}\\
h^{1,1,1}\left((\mathbb{C}^{\ast})^2\right)=2, &  \text{weight 2 in the }1^{st}\text{-cohomology and}\\   
h^{2,2,2}\left((\mathbb{C}^{\ast})^2\right)=1, &  \text{weight 4 in the }2^{nd}\text{-cohomology}.
\end{eqnarray*}
Observe finally that, being $(\mathbb{C}^{\ast})^2$ a non-singular variety, its weights are greater or equal than the order of the cohomology, as pointed out in Remark \ref{rem:weights}.
\end{Example}

\begin{Exercise}
Use the argument of Example \ref{ex:2structures} to find two different pure or mixed Hodge structures in the cotangent bundle $T^{\ast}\Pic^0_{\Sigma_g}$ where $\Pic^0_{\Sigma_g}$ is the jacobian variety of a compact Riemann surface of genus $g$, and its diffeomorphic variety $(\mathbb{C}^{\ast})^{2g}$.
\end{Exercise}

\section{The \texorpdfstring{$e$}{e}-polynomial and its basic properties}

There are a number of geometric invariants that we can construct by combining the mixed Hodge numbers $h^{k,p,q}=\dim_{\mathbb{C}} H^{k,p,q}(X)$ and the compactly supported ones $h_c^{k,p,q}=\dim_{\mathbb{C}} H_c^{k,p,q}(X)$.

\begin{Definition}\label{def:mHpolynomials}
Let $X$ be a complex variety, not necessarily smooth or projective. We define the \textbf{mixed Hodge polynomial} of $X$ as the polynomial on three variables given by
\begin{equation}
\mu(X;\,t,u,v):=\sum_{k,p,q}h^{k,p,q}(X)\ t^{k}u^{p}v^{q}\in\mathbb{\mathbb{N}}_{0}[t,u,v]\; .\label{eq:mu}
\end{equation}
Analogously, define the \textbf{compactly supported mixed Hodge polynomial} by
\begin{equation}
\mu_c(X;\,t,u,v):=\sum_{k,p,q}h_c^{k,p,q}(X)\ t^{k}u^{p}v^{q}\in\mathbb{\mathbb{N}}_{0}[t,u,v]\; .\label{eq:muc}
\end{equation}
\end{Definition}
Observe that mixed Hodge polynomials specialize to the
corresponding Poincaré polynomial $P(X)$ and compactly supported Poincaré polynomial $P_c(X)$), respectively, by setting $u=v=1$ in $P(X)=\mu(X;\,t,1,1)$ and
$P_{c}(X)=\mu_c(X;\,t,1,1)$.

By substituting $t=-1$, compactly supported mixed Hodge polynomials become a very useful
generalization of the Euler characteristic, called the $e$-polynomial of $X$.
\begin{Definition}\label{def:epolynomials}
Let $X$ be a complex variety, not necessarily smooth or projective. We define the \textbf{$e$-polynomial} of $X$ as the polynomial on two variables given by
\begin{equation}\label{eq:def_e}
e(X;\,u,v):=\sum_{k,p,q}(-1)^{k}h_c^{k,p,q}(X)\ u^{p}v^{q}\in\mathbb{Z}[u,v].
\end{equation}
\end{Definition}
Note that from the $e$-polynomial we can compute the \textbf{compactly supported Euler characteristic} of
$X$ as 
\begin{equation}
\chi_c(X)=e(X;\,1,1)=\mu_c(X;\,-1,1,1)\; .
\end{equation}
Observe that the compactly
supported Euler characteristic equals the usual one for complex quasi-projective
varieties. 

Defining the \textbf{$(p,q)$-(compactly supported) Euler characteristic} as 
\begin{equation}
\chi_c^{p,q}(X)=\sum_k(-1)^k h_c^{k,p,q}(X)\; ,
\end{equation}
we can express the $e$-polynomial as a two variable polynomial whose coefficients are these $(p,q)$-(compactly supported) Euler characteristics: 
\begin{equation}
e(X;\,u,v):=\sum_{p,q}\chi_c^{p,q}(X)\ u^{p}v^{q}\in\mathbb{Z}[u,v].
\end{equation}

\begin{Definition}\label{def:hodge-tate}
A complex variety $X$ is said to be of \textbf{Hodge-Tate type} or 
\textbf{balanced} type if all its $k$-weights are of the form
$(p,p)$ with $p\in\{0,\ldots,k\}$. 
\end{Definition}
Complex affine algebraic groups and smooth toric varieties are examples of balanced varieties. When computing $e$-polynomials of balanced varieties we will use the notation $x:=uv$ and $e(X; x):=e(X; uv)$, yielding polynomials in one variable.

\begin{Remark}\label{rem:irreducible}
    If $X$ is connected, non-singular and balanced, then there is a well-defined notion of highest-degree monomial in their compactly supported mixed Hodge and $e$-polynomials, given by 
    \[h^{k,p,p}(X)t^ku^pv^p=h^{k,p,p}(X)t^kx^p,\] 
    where $2p\geq k$. This is the leading term of the $e$-polynomial and the leading coefficient is the top mixed Hodge number $h^{k,p,p}(X)$ telling us how many irreducible components $X$ has. Therefore, if the $e$-polynomial of $X$ is monic, $X$ is \textbf{irreducible}. 
\end{Remark}

The following shows a useful relationship between mixed Hodge polynomials and $e$-polynomials for smooth varieties. 

\begin{Proposition}\cite[Corollary 2.1.5]{hausel2008mixed}\label{prop:relation_pols}
    Let $X$ be a smooth connected complex variety of complex dimension $d$. The following equality holds between mixed Hodge polynomials
    \begin{equation}
        \mu_c(X; t,u,v)=(t^2uv)^d\mu\left(\frac{1}{t},\frac{1}{u},\frac{1}{v}\right)\; .
    \end{equation}
    Therefore, we have this relationship between the $e$-polynomial and the mixed Hodge polynomial
       \begin{equation}
        e(X; u,v)=(uv)^d\mu\left(-1,\frac{1}{u},\frac{1}{v}\right)\; .
    \end{equation}
\end{Proposition}

\begin{Exercise}
 Prove Proposition \ref{prop:relation_pols} (c.f. \cite[Corollary 2.1.5]{hausel2008mixed}). 
\end{Exercise}

Let us compute these polynomials in simple cases. 

\begin{Example}
   The variety $\mathbb{C}^{\ast}$ is smooth and connected of complex dimension $1$. From Example \ref{ex:2structures} we know its mixed Hodge numbers, hence we have that its mixed Hodge polynomial is 
\begin{equation}\label{eq:muc_C*}
\mu(\mathbb{C}^{\ast}; t,u,v)=h^{0,0,0}t^0u^0v^0+h^{1,1,1}t^1u^1v^1=1+tuv.
\end{equation}
   By Proposition \ref{prop:relation_pols} we get
\begin{equation}\label{eq:mu_C*}
\mu_c(\mathbb{C}^{\ast}; t,u,v)=(t^2uv)\cdot \mu\left(\mathbb{C}^{\ast}; \frac{1}{t}, \frac{1}{u},\frac{1}{v}\right)=(t^2uv)\left(1+\frac{1}{tuv}\right)=t+t^2uv\; ,
\end{equation}
and 
\begin{equation}\label{eq:eC*}
e(\mathbb{C}^{\ast}; u,v)=\mu_c(\mathbb{C}^{\ast}; -1,u,v)=-1+uv\; .
\end{equation}
Being $\mathbb{C}^{\ast}$a balanced variety, we can rewrite its $e$-polynomial as
\begin{equation}\label{eq:eC*balanced}
e(\mathbb{C}^{\ast}; x)=\mu_c(\mathbb{C}^{\ast}; -1,u,v)=-1+x\; .
\end{equation}
\end{Example}

The e-polynomials satisfy a number of good properties which makes them suitable for computations. 

\begin{Proposition}\label{prop:multiplicative_pols}
Mixed Hodge and $e$-polynomials satisfy a multiplicative
property with respect to Cartesian products, i.e. 
\[\mu(X\times Y)=\mu(X)\cdot \mu(Y)\; ,\]
\[\mu_c(X\times Y)=\mu_c(X)\cdot \mu_c(Y)\; ,\]
\[e(X\times Y)=e(X)\cdot e(Y)\; .\]
\end{Proposition}
\begin{proof}
This comes from the fact that K\"unneth isomorphism $H^{\ast}(X\times Y)\simeq H^{\ast}(X)\otimes H^{\ast}(Y)$ is compatible with mixed Hodge structures, see \cite{peters2008ergebnisse, hausel2008mixed}. 
\end{proof}

Moreover, compared to mixed Hodge polynomials and Poincaré polynomials, $e$-polynomial are additive over locally closed stratifications, because of the following: 
\begin{Proposition}\cite[Théorème 8.3.9]{deligne1974theorie}\label{prop:e-strat}
Let $X$ be a complex variety and let $Z\subset X$ be a closed subvariety. Then we have 
\[
e(X)=e(Z)+e(X\backslash Z),
\]
and also
\[\chi_c(X)=\chi_c(Z)+\chi_c(X\backslash Z).\] 
\end{Proposition}

One of the crucial facts about the good behavior of $e$-polynomials, allowing to compute them in many cases, is its multiplicativity under certain fibrations. Let $X,B$ be quasi-projective varieties and let $\pi:X\rightarrow B$ be an algebraic morphism which is an algebraic fibration, i.e. where all fibers are isomorphic to an algebraic variety $F$. 

\begin{Theorem}\label{thm:fibrations}
Let $\pi:X\rightarrow B$ be an algebraic fibration with fiber $F$ as above, such that all three spaces $X,B,F$ are smooth, the fibration is locally trivial in the analytic topology and the fundamental group of the base $\pi_1(B)$ acts trivially on $H_c^{\ast}(F)$, the compactly supported cohomology of the fiber. Then, 
\[e(X)=e(F)\cdot e(B)\; .\]
\end{Theorem}
\begin{proof}
See \cite[Proposition 2.4]{logares2013hodge} or \cite[Theorem 6.1]{dimca1997purity}.
\end{proof}

\begin{Remark}\label{rem:fibrations}
    The hypothesis of Theorem \ref{thm:fibrations} are satisfied in many relevant situations where computations of $e$-polynomials arise:
\begin{enumerate}
\item[(a)] If the fibration is locally trivial in the Zariski topology of the base $B$, and $B$ is irreducible (\cite[Remark 2.5]{logares2013hodge}). 
\item[(b)] If $F$ is a complex connected algebraic group and $\pi$ is a principal $F$-bundle.
\item[(c)] In particular, if $F$ is \textbf{special} \cite{grothendieck1958special}, which means that all principal $F$-bundles are locally Zariski trivial.
\item[(d)] If $X=G$ is an algebraic reductive group, its center $F=Z(G)$ is connected and the base is the adjoint group $B=PG=G/Z(G)$ (\cite[Proposition 2.6]{florentino2021serre}).
\end{enumerate}
\end{Remark}

Let us explore these properties to compute $e$-polynomials.

\begin{Example}
\label{ex:sigma_g}
Let $\Sigma_g$ be a genus $g$ compact Riemann surface. In this case, compactly supported cohomology equals usual cohomology and the (pure) Hodge numbers were given in Example \ref{ex:diamond_RS}. Then, the different invariants are
\[\mu(\Sigma_g;t,u,v)=\mu_c(\Sigma_g;t,u,v)=1+gt(u+v)+t^2uv\; ,\]
\[P(\Sigma_g;t)=1+2gt+t^2\; ,\]
\[e(\Sigma_g;u,v)=1-g(u+v)+uv\; ,\]
\[\chi(\Sigma_g)=2-2g\; .\]
Observe that a compact Riemann surface $\Sigma_g$ carry a pure but not balanced Hodge structure, because there are weights $h^{1,1,0}=h^{1,0,1}\neq 0$ which are not of the form $h^{k,p,p}$. As a consequence, its $e$-polynomial cannot be expressed as a one-variable polynomial. 
\end{Example}

\begin{Example}
Using the Hodge numbers described in Example \ref{ex:diamond_PS}, the invariants of the projective space are 
\[\mu(\mathbb{P}^n_{\mathbb{C}};t,u,v)=\mu_c(\mathbb{P}^n_{\mathbb{C}};t,u,v)=1+t^2uv+t^4u^2v^2+\cdots +t^{2n}u^nv^n\; ,\]
\[P(\mathbb{P}^n_{\mathbb{C}};t)=1+t^2+t^4+\cdots +t^{2n}\; ,\]
\[e(\mathbb{P}^n_{\mathbb{C}};u,v)=1+uv+u^2v^2+\cdots + u^nv^n=1+x+x^2+\cdots + x^n\; ,\]
\[\chi(\mathbb{P}^n_{\mathbb{C}})=n+1\; .\]
Observe that the projective space carries a pure and also balanced Hodge structure. 
\end{Example}

\begin{Example}
The locally closed decomposition $\mathbb{C}=\mathbb{C}^{\ast}\sqcup \{pt\}$ yields the equality 
\[e(\mathbb{C};u,v)=e(\mathbb{C}^*;u,v)+e(\{pt\};u,v)=(uv - 1)+1=(x-1)+1=x\; ,\] between the corresponding $e$-polynomials. 
\end{Example}

\begin{Example}
    By the computations in Example \ref{ex:2structures}, the mixed Hodge polynomial of $(\mathbb{C}^{\ast})^2$ is 
    \[\mu((\mathbb{C}^{\ast})^2; t,u,v)=1+2tuv+t^2u^2v^2.\]
    Using Proposition \ref{prop:relation_pols}, we obtain the compactly supported mixed Hodge polynomial
    \[\mu_c((\mathbb{C}^{\ast})^2; t,u,v)=(t^2uv)^2\mu((\mathbb{C}^{\ast})^2; t,u,v)=\]
    \[(t^2uv)^2\left(1+\frac{2}{tuv}+\frac{1}{t^2u^2v^2}\right)=t^2+2t^3uv+t^4u^2v^2.\]
    From this, we can extract the compactly supported mixed Hodge numbers of $(\mathbb{C}^{\ast})^2$, obtaining $h_c^{2,0,0}\left((\mathbb{C}^{\ast})^2\right)=1$, $h_c^{3,1,1}\left((\mathbb{C}^{\ast})^2\right)=2$, $h_c^{4,2,2}\left((\mathbb{C}^{\ast})^2\right)=1$. Now, we can apply the definition (\ref{eq:def_e}) to get the $e$-polynomial:
    \[e((\mathbb{C}^{\ast})^2; u,v)=\mu_c((\mathbb{C}^{\ast})^2;-1,u,v)=1-2uv+u^2v^2.\]
    We observe that $(\mathbb{C}^{\ast})^2$ is a balanced variety, then we can express its $e$-polynomial in one variable:
    \[e((\mathbb{C}^{\ast})^2; x)=1-2x+x^2.\]
    Finally, note that Proposition \ref{prop:multiplicative_pols} holds in this case:
    \[\mu((\mathbb{C}^{\ast})^2; t,u,v)=1+2tuv+t^2u^2v^2=(1+tuv)^2=\mu(\mathbb{C}^{\ast}; t,u,v)\cdot \mu(\mathbb{C}^{\ast}; t,u,v)\; ,\]
    \[\mu_c((\mathbb{C}^{\ast})^2; t,u,v)=t^2+2t^3uv+t^4u^2v^2=(t+t^2uv)^2=\mu_c(\mathbb{C}^{\ast}; t,u,v)\cdot \mu_c(\mathbb{C}^{\ast}; t,u,v)\; ,\]
    \[e((\mathbb{C}^{\ast})^2; t,u,v)=1-2uv+u^2v^2=(-1+uv)^2=e(\mathbb{C}^{\ast}; t,u,v)\cdot e(\mathbb{C}^{\ast}; t,u,v).\]
\end{Example}

Observe that all these varieties $\Sigma_g$, $\mathbb{P}^n_{\mathbb{C}}$, $\mathbb{C}^{\ast}$, $(\mathbb{C}^{\ast})^2$ are irreducible, which can be seen from their $e$-polynomials (see Remark \ref{rem:irreducible}). 

\begin{Exercise}
    Compute the mixed Hodge, compactly supported mixed Hodge, $e$- and Poincaré polynomials of the maximal torus of the complex group $\GL_n$. 
\end{Exercise}

\begin{Example}
Let us compute the $e$-polynomial of $\GL_2:=\GL_2(\mathbb{C})$ by using the property on fibrations. On the one hand, we have the surjection
\[\GL_2\twoheadrightarrow \mathbb{C}^2\backslash \{(0,0)\}, \;\; \left(\begin{array}{cc}
a & b \\ 
c & d
\end{array} \right)\mapsto (a,c)\; ,\] 
where fibers are given by vectors $(b,d)$ linearly independent with $(a,c)$, i.e. fibers are isomorphic to $\mathbb{C}^2\backslash \mathbb{C}$. Then, this is a locally Zariski trivial fibration 
\[\mathbb{C}^2\backslash \mathbb{C}\hookrightarrow\GL_2\twoheadrightarrow \mathbb{C}^2\backslash \{(0,0)\}\; ,\]
where the $e$-polynomials of the base and the fiber can be computed by Propositions \ref{prop:multiplicative_pols} and \ref{prop:e-strat}:
\[e(\mathbb{C}^2\backslash \mathbb{C};u,v)=e(\mathbb{C}^2;u,v)-e(\mathbb{C};u,v)=
e(\mathbb{C};u,v)\cdot e(\mathbb{C};u,v)-e(\mathbb{C};u,v)=\]
\[(uv)\cdot (uv) - uv = u^2v^2-uv\]
\[\text{and}\quad\quad e(\mathbb{C}^2\backslash \{(0,0)\};u,v)=e(\mathbb{C}^2;u,v)-e(\{(0,0)\};u,v)=\]
\[e(\mathbb{C};u,v)\cdot e(\mathbb{C};u,v)-e(\{(0,0)\};u,v)=(uv)\cdot (uv) - 1=u^2v^2-1.\]
Therefore, by Theorem \ref{thm:fibrations} and Remark \ref{rem:fibrations}, 
\[e(\GL_2;u,v)=e(\mathbb{C}^2\backslash \mathbb{C};u,v)\cdot e(\mathbb{C}^2\backslash \{(0,0)\};u,v)=(u^2v^2-uv)(u^2v^2-1)\; ,\]
which can be expressed as the $e$-polynomial in one variable for a balanced variety:
\[e(\GL_2;x)=e(\mathbb{C}^2\backslash \mathbb{C};x)\cdot e(\mathbb{C}^2\backslash \{(0,0)\};x)=
(x^2-x)\cdot (x^2-1)=x^4-x^3-x^2+x.\]

\end{Example}

\begin{Example}
For a fibration to satisfy the multiplicative property on $e$-polynomials, hypotheses in Theorem \ref{thm:fibrations} and Remark \ref{rem:fibrations} are quite restrictive. For example, consider a double cover $X\rightarrow \Sigma_1$ of an elliptic curve (i.e. a Riemann surface of genus $1$) with two branching points, what prevents it from being locally trivial. By the Riemann-Hurwitz formula, the genus of the covering space $g(X)$ is related to the genus of the base $g(\Sigma_1)=1$ by the formula
\[2g(X)-2=n(2g(\Sigma_1)-2)+\sum_{branch.}(e_i-1)\]
where $e_1=e_2=1$ and $n=2$ is the degree of the covering. Hence, 
\[2g(X)-2=2(2\cdot 1 - 1)+(1+1)\Rightarrow g(X)=2\]
therefore $X=\Sigma_2$, a Riemann surface of genus $2$. Using the computations in Example \ref{ex:sigma_g}, the $e$-polynomials of the total space and the base of this covering are, 
\[e(\Sigma_2)=1-2(u+v)+uv\quad \text{and}\quad e(\Sigma_1)=1-(u+v)+uv.\]
Observe that the latter polynomial does not divide the former, hence the $e$-polynomial does not satisfy any sort of multiplicative property.  Although the generic fiber is two points, whose $e$-polynomial is $2$, it is false that $2\cdot e(\Sigma_1)=e(\Sigma_2)$. 
\end{Example}

\begin{Exercise}
    Compute the $e$-polynomial of the complex groups  $\SL_n$ and $\PGL_n$.
\end{Exercise}

\begin{Remark}\label{rem:motives}

Let ${\rm KVar}$ denote the ring generated by isomorphism classes $[X]$ of algebraic varieties modulo cut-and-paste relations $[X]=[Y]+[X-Y]$, for $Y\subseteq X$ a closed subvariety, and where multiplication in the ring is given by $[X\times Z]=[X]\cdot [Z]$. The elements in this \textbf{Grothendieck ring of algebraic varieties} are called \textbf{virtual classes} or \textbf{motives} of the varieties. It can be shown that the $e$ polynomial factors through the Grothendieck ring $e:{\rm KVar}\rightarrow \mathbb{Z}[u,v]$, making the $e$-polynomial a coarser invariant for complex algebraic varieties, than what the motive is. 
\end{Remark}

\section{Techniques to compute \texorpdfstring{$e$}{e}-polynomials of character varieties}

In this section we introduce character varieties, which are affine algebraic varieties playing a prominent role in algebraic geometry, and we present a combination of arithmetic and geometric techniques to compute their character varieties. 

\subsection{Character varieties}

Let us start by defining the notion of representation variety. 

\begin{Definition}\label{def:rep_var}
 Let $G$ be a reductive algebraic group over $\mathbb{C}$. Let $\Gamma = \langle \gamma_1, \ldots, \gamma_s\; : \; r_{t}(\gamma_i)=1\rangle$ be a finitely generated group. 
We define the \textbf{$G$-representation variety} as 
\[\mathcal{R}_{G}(\Gamma):=\Hom (\Gamma, G)=\left\{\rho(\gamma) = \big(\rho(\gamma_{1}),\ldots,\rho(\gamma_{s})\big)\in G^s : r_{t}(\rho(\gamma))=1 \right\},\]
which is an affine algebraic variety.   
\end{Definition}

There is an action of $G$ on $\mathcal{R}_{G}(\Gamma)$ by conjugation. 
For $\rho\in \mathcal{R}_{G}(\Gamma)$, $g\in G$, $\gamma\in \Gamma$, we have 
\[(g\cdot \rho) (\gamma) := g \rho(\gamma) g^{-1}=\big(g\rho(\gamma_{1})g^{-1},\ldots,g\rho(\gamma_{s})g^{-1}\big)\, .\]

The next definition presents character varieties as Geometric Invariant Theory (GIT) quotients. Readers not familiar with GIT can consult the seminal works \cite{mumford1994geometric, newstead2012introduction}. Essentially, affine GIT combines together orbits in the quotient whose closures intersect (hence the double slash), in order to have a Hausdorff quotient and such that the ring of functions of the quotient equals the ring of $G$-invariant functions in the domain: this way invariant theory turns out to have a geometrical meaning. 

\begin{Definition}\label{def:CV}
  We define the \textbf{$G$-character variety} of $\Gamma$ as the affine GIT quotient of the $G$-representation variety by the conjugation action:
\[\mathcal{X}_{G}(\Gamma):= \mathcal{R}_{G}(\Gamma)/\!\!/G =
\Spec \mathbb{C}[\mathcal{R}_{G}(\Gamma)]^{G}.\]  
\end{Definition}

Let us also denote by $\mathcal{R}^{\ast}_{G}(\Gamma)$ the set of \textbf{irreducible representations} and \linebreak $\mathcal{X}_{G}^{\ast}(\Gamma):=\mathcal{R}^{\ast}_{G}(\Gamma)/G$ the corresponding GIT quotient, which is a geometric quotient or an orbit space, where each point corresponds to an orbit. 

The name character varieties comes from the fact that these varieties are generated by characters. 
Given a representation $\rho\in\mathcal{R}_{G}(\Gamma)$, the character of the representation is
\[\chi_{\rho}:\Gamma\longrightarrow \mathbb{C}\; , \quad \gamma\mapsto \tr (\rho(\gamma))\; .\]
Artin \cite{artin1969azumaya} conjectures that the invariants of complex matrices under simultaneous conjugations are polynomials in $\tr(X_{i_1}\cdots X_{i_s})$. This was proved by Procesi in \cite{procesi1976invariant}. Basic examples of this are the isomorphism of algebras
\[\mathbb{C}[\GL_n]^{\GL_n} \simeq \mathbb{C}[c_1, \ldots , c_n^{\pm 1} ]\; ,\]
where $c_0 = 1$, $c_1 = \tr X$, ..., $c_n = \det X$ are the coefficients of the characteristic polynomial of $X$. A more elaborated one says that the invariants of pairs $(A,B)$ of $2\times 2$ matrices under simultaneous conjugation
are generated by the five traces $\tr A$, $\tr A^2$, $\tr B$, $\tr B^2$, $\tr AB$. The classical result of Fricke and Klein \cite{klein1890elliptic} states that the ring of $\SL_2$-invariant pairs of matrices is the ring of polynomials generated by the traces of both matrices and the trace of the product, this is
\[\mathbb{C}[\mathcal{R}_{\SL_2}(F_2)]^{\SL_2}=\mathbb{C}[\tr A, \tr B, \tr AB]\; ,\]
where $F_r$ denotes the free group in $r$ generators. From this, we obtain that the character variety $\mathcal{X}_{\SL_2}(F_2)$ is isomorphic to $\mathbb{C}^3$.

When the finitely generated group $\Gamma$ is $\pi_{1}(\Sigma_{g})$, the fundamental group of a genus $g$ Riemann surface $\Sigma_{g}$, character varieties are related to moduli spaces of Higgs bundles through the \textbf{non-abelian Hodge correspondence} (\cite{hitchin1987, Donaldson87, Corlette88, simpson}) stating that the character variety (sometimes called the \textbf{Betti moduli space}) $\mathcal{X}_G(\Gamma)=\mathcal{R}_G(\Gamma)/\!\!/G $ is diffeomorphic (but not complex algebraic isomorphic) to  the Dolbeault moduli space of $G$-Higgs bundles over $\Sigma_g$.

\subsection{Arithmetic methods to compute \texorpdfstring{$e$}{e}-polynomials}

Here we will recall the ideas behind the use of arithmetics to compute $e$-polynomials as polynomials counting the number of points of a variety over a finite field. This is the strategy used in \cite{hausel2008mixed} to compute the $e$-polynomial of the character varieties $\mathcal{X}_{\GL_n}(\Gamma)$ for $\Gamma=\pi_1(\Sigma_g)$, the fundamental group of a genus $g$ Riemann surface. It is based on a particular property of certain varieties for which there exists a polynomial counting the number of points over almost each finite field, and which coincides with the $e$-polynomial. This method is inspired in the Weil conjectures. 

Let $\mathbb{F}_{q}$ be a 
finite field with $q$ elements and characteristic $p$, so that $q=p^{s}$, $s\in\mathbb{N}$. A scheme $X$, defined over $\mathbb{Z}$,
is called of \textbf{polynomial type} if there exists a polynomial $C_{X}(t)\in\mathbb{Z}[t]$
(called the \textbf{counting polynomial} for $X$) such that the number
of $\mathbb{F}_{q}$-points of $X$ is given by the evaluation of the counting polynomial at $q$: 
\[
|X/\mathbb{F}_{q}|=C_{X}(q)\; ,
\]
for every $s$ and almost every prime $p$ (i.e. for every prime except for a finite number of them). 

\begin{Theorem}\cite[Appendix]{hausel2008mixed}\label{thm:Katz}
Let $X$ be a $\mathbb{Z}$-scheme of polynomial type with counting polynomial $C_{X}$. Then, $X$ is balanced and the $e$-polynomial of the complex variety $X(\mathbb{C}):=X\otimes_{\mathbb{Z}}\mathbb{C}$
coincides with the counting polynomial: 
\[
e(X(\mathbb{C});x)=C_{X}(x)\; .
\]
\end{Theorem}

\begin{Exercise}
    Check that the number of points of $\mathbb{C}$, $\mathbb{C}^{\ast}$ and $\mathbb{P}^n_{\mathbb{C}}$ over finite fields is actually computed by the $e$-polynomial, i.e. the $e$-polynomial is a (the) counting polynomial for these varieties. Indeed, these varieties are of polynomial type. 
\end{Exercise}

In \cite[Theorem 3.5.1]{hausel2008mixed} the authors compute the $e$-polynomial of the character varieties $\mathcal{X}_{\GL_2}(\pi_1(\Sigma_g))$ with this method.
To compute the counting polynomial, it is used the character formula \cite[Proposition 2.3.2]{hausel2008mixed} counting the number of solutions of equations in finite groups. This method has been shown to be successful in \cite{mereb2015SLntwisted} for $\SL_n$ and in \cite{baraglia2017arithmetic} for $\GL_n$, $\SL_n$, $n=2,3$, simplifying the calculations. Computations for a general $\Gamma$ and $\GL_n$, $n\geq 4$ become intractable. 

Let us describe explicitly this method in a particular and simpler example, following \cite{mozgovoy2015arithmetic}, where the counting polynomial of the $\GL_n$-character variety of the free group is obtained via counting irreducible representations. 

Let $\Gamma=F_r$ be the free group in $r$ generators and let $G$ be the complex general linear group $\GL_n(\mathbb{C})$. Let $\mathbb{F}_q$ the finite field of $q$ elements, $q=p^s$, $p$ a prime number. Let 
\[\mathcal{X}_{\GL_n}(F_r)=\GL_n(\mathbb{C})^r/\!\!/ \GL_n(\mathbb{C})\] 
be the $\GL_n$-character variety of the free group in $r$ generators. 

We define the $\mathbb{Z}$-scheme
\[\mathcal{X}_{n,r}=\Spec(\mathbb{Z}[\Hom (F_r, \GL_n(\mathbb{Z}))]^{\GL_n(\mathbb{Z})})=\Spec(\mathbb{Z}[\GL_n(\mathbb{Z}))^r]^{\GL_n(\mathbb{Z})}).\]
For each $q$, denote by 
\[A_{n,r}(q):=\mathcal{X}_{n,r}(\mathbb{F}_q)\; ,\]
\[A^{\ast}_{n,r}(q):=\mathcal{X}^{\ast}_{n,r}(\mathbb{F}_q),\]
respectively, the number of representations and irreducible representations of these character varieties over $\mathbb{F}_q$. By \cite[Appendix]{hausel2008mixed} these $A_{n,r}$ and $A_{n,r}^{\ast}$ are the counting polynomials for the character varieties $\mathcal{X}_{\GL_n}(F_r)$ and $\mathcal{X}^{\ast}_{\GL_n}(F_r)$.

Let us define some operators in the power series ring $\mathbb{Q}[q][[t]]$, whose maximal ideal is $t\mathbb{Q}[q][[t]]$. Define the \textbf{Adams operator} $\Psi$ as 
\begin{equation}\label{eq:Adams}
\Psi:\mathbb{Q}[q][[t]]\rightarrow \mathbb{Q}[q][[t]]\; , \quad \Psi(q^i t^n)=\sum_{m\geq 1} \frac{q^{im} t^{nm}}{m}\; ,
\end{equation}
and extend it by $\mathbb{Q}$-linearity, whose inverse is given by 
\begin{equation}\label{eq:Adams-1}
\Psi^{-1}(q^i t^n)=\sum_{m\geq 1}\frac{\mu(m) q^{im}t^{nm}}{m}\; .
\end{equation}
Here $\mu:\mathbb{N}\rightarrow \{-1,0,1\}$ is the M\"obius function such that $\mu(m)=(-1)^k$ if $m$ is square-free and have $k$ primes in its arithmetic decomposition, and $\mu(m)=0$ otherwise. 

\begin{Exercise}
    Show that the inverse of the Adams operator defined on monomials as $\Psi(q^i t^n)=\sum_{m\geq 1} \frac{q^{im} t^{nm}}{m}$ is the operator defined on monomials as $\Psi^{-1}(q^i t^n)=\sum_{m\geq 1}\frac{\mu(m) q^{im}t^{nm}}{m}$.
\end{Exercise}

Define the \textbf{plethystic exponential} $\PExp$ by 
\begin{equation}\label{eq:PExp}
  \PExp: t\cdot \mathbb{Q}[q][[t]]\rightarrow 1+t\mathbb{Q}[q][[t]]\;, \quad   \text{where }\PExp(q^i t^n) = (1-q^i t^n)^{-1}
\end{equation}
on each monomial and extend it by $\PExp(f+g)=\PExp(f)\PExp(g)$. The inverse of this map is the \textbf{plethystic logarithm} denoted by 
\begin{equation}\label{eq:PLog}
    \PLog: 1+t\mathbb{Q}[q][[t]]\rightarrow \mathbb{Q}[q][[t]]\; .
\end{equation} 
Given an element $f(q,t)=1+\sum_{n\geq 1}f_n(x)t^n\in \mathbb{Q}[q][[t]]$, the plethystic operators relate to the Adams operator by 
\begin{equation}\label{eq:pletyhistic_Adams}
\PExp(f)=e^{\Psi(f)}\quad,\quad \PLog(f)=\Psi^{-1}(\log f)\; .
\end{equation}
Define the shift operator $S$ on $\mathbb{Q}[q][[t]]$ by
\begin{equation}\label{eq:shift}
    S (t^n) = q^{(r-1)\binom{n}{2}}t^n\; ,
\end{equation}
and define the power series 
\begin{equation}\label{eq:F}
F(t)=\sum_{n\geq 1}\big((q-1)(q^{2}-1)\ldots(q^{n}-1)\big)^{r-1}\,t^{n}.
\end{equation}

We want to compute explicitly the number of isomorphism classes of absolutely
irreducible representations of the free group in $r$ generators, in order to obtain a generating series for these quantities. 
Given that representations of the group algebra $k\Gamma=\mathbb{F}_q F_r$ are equivalent to representations of the path algebra of the quiver with one vertex and $r$ loops, we will use the theory of quiver representations in \cite{mozgovoy2009number}. 

Define the \textbf{Hall algebra} of the group algebra $\mathbb{F}_q F_r$ by 
\[H_{q,r}=\prod_{[V]} \mathbb{Q}\cdot V\; ,\]
 where the product is taken over all isomorphism classes of finite-dimensional
representations $V$ of $\mathbb{F}_q F_r$. In this Hall algebra, the grading is given by the dimension and the product is defined by 
\[[V] \cdot [W] =
\sum_{[X]}g_{V,W,X}[X]\; ,\] 
where $g_{V,W,X}$ is the number of subrepresentaions $U\subset X$ such that $U\simeq V$ and $X/U\simeq W$. This makes 
$H_{q,r}$ a $\mathbb{Z}_{\geq 0}$-graded complete
local associative unital $\mathbb{Q}$-algebra, where elements with constant
term $1$ (the class of the zero-dimensional representation) are invertible. 

Define the evaluation map
\begin{equation}
    I: H_{q,r}\rightarrow \mathbb{Q}[[t]]\;, \quad [V]\mapsto \frac{t^{\dim V}}{|\Aut(V)|}
\end{equation}
which, composed with the inverse of the shift operator (\ref{eq:shift}) gives a homomorphism of $\mathbb{Q}$-algebras $S^{-1}\circ I: H_{q,r}\rightarrow \mathbb{Q}[[t]]$ (c.f. \cite[Lemma 3.4]{mozgovoy2009number}).

Recall that $\Hom (F_r, \GL_n(\mathbb{F}_q))=(\GL_n(\mathbb{F}_q))^r$. To compute the number of points of $\GL_n(\mathbb{F}_q)$ we observe that the first row of the matrix can have whichever values except for the zero vector, then there are $q^n-1$ possibilities. For each one, the second row can take whichever value except for a multiple of the first row, then having $q^n-q$ possibilities. Given the first two rows, the third row can take whichever value except for a linear combination of the first two rows (which are $q^2$ linear combinations), then having $q^n-q^2$ possibilities. Finally we get
\begin{equation}\label{eq:GLnFq}
    |\GL_n(\mathbb{F}_q)| = \prod_{i=0}^{n-1} (q^n-q^i) = q^{\binom{n}{2}}\prod_{i=1}^n(q^i-1)\; .
\end{equation}

\begin{Exercise}
    List the elements of $\GL_n(\mathbb{F}_q)$ for lower $n$ and $q$, and check that its cardinality coincides with the value of $e(\GL_n;x)(q)$. 
\end{Exercise}

Denote by $z$ the element $z=\sum_{[V]} [V]\in H_{q,r}$, the sum of all (isomorphism classes of) representations in the Hall algebra, and observe that this element is invertible because it contains the zero-dimensional representation which is the unit in $H_{q,r}$. Then we obtain the expression
\
\begin{equation}\label{eq:Iz}
    I(z)=\sum_{n\geq 0}\frac{|\GL_n(\mathbb{F}_q)|^r}{|\GL_n(\mathbb{F}_q)|}t^n=\sum_{n\geq 0}\big( q^{\binom{n}{2}}\prod_{i=1}^n(q^i-1)\big)^{r-1} t^n=(S\circ F)(t)=S(F(t))\; ,
\end{equation}
where $S$ and $F$ are defined in (\ref{eq:shift}), (\ref{eq:F}). Given that this implies $(S^{-1}\circ I)(z)=F(t)$ and $S^{-1}\circ I$ is a homomorphism we get $(S^{-1}\circ I)(z^{-1})=F(t)^{-1}$ and 
\begin{equation}\label{eq:Iz-1}
     I(z^{-1})=S\left(F(t)^{-1}\right)\; .
\end{equation}

Let us write $z^{-1}=\sum_{[V]}\gamma_V [V]$. Observe that 
\[H_{q,r}\ni 1=z\cdot z^{-1}=\sum_{[V]}[V]\cdot \sum_{[V]}\gamma_V [V]=\sum_{[X]}\big(\sum_{U\subset X}\gamma_{U}\big)[X]\; ,\]
where the inner sum is taken over all possible subrepresentations in a isomorphism class $[X]$. Therefore $\sum_{U\subset X}\gamma_U=0$, unless $[X]$ is the zero representation. Note that if $V$ is not completely reducible then $\gamma_V=0$, otherwise it cannot cancel in the equality $1=z\cdot z^{-1}$. 
Now we apply \cite[Lemma 3.5]{mozgovoy2009number} to compute the coefficients $\gamma_V$. 

\begin{Lemma}\cite[Lemma 3.5]{mozgovoy2009number}
Let $\mathcal{S}$ be the set of isomorphism classes of irreducible representations of $\mathbb{F}_qF_r$. If $V\simeq \oplus_{[S]\in\mathcal{S}}S^{m_S}$ is completely reducible, then 
\[\gamma_{V}=\prod_{[S]\in\mathcal{S}}(-1)^{m_S}|\End(S)|^{\binom{m_S}{2}}\; .\]
\end{Lemma}
\begin{proof}
  Let $V\simeq \oplus_{[S]\in\mathcal{S}}S^{m_S}$ be a completely reducible representation. A polystable sub-representation $U\subset V$ (otherwise $\gamma_U=0$) is of the form $U\simeq \oplus_{[S]\in \mathcal{S}} S^{a_S}$, where $0\leq a_S\leq m_S$ for each $[S]\in\mathcal{S}$. Fixing the tuple $(a_S)_{[S]\in\mathcal{S}}$, the number of sub-representations $U\subset V$ with that isotypical decomposition is the cardinality of this product of Grassmanninans:
  \begin{equation}\label{eq:lemma}
  \#\big(\prod_{[S]\in\mathcal{S}}\Gr_{a_S}\big( \End(S)^{m_S}\big)\big)=\prod_{[S]\in\mathcal{S}}\genfrac{[}{]}{0pt}{}{m_S}{a_S}_{|\End(S)|}\; ,
  \end{equation}
  where $\genfrac{[}{]}{0pt}{}{m}{n}_q=\frac{[m]_q^!}{[n]_q^! [m-n]_q^!}$, and the quantum factorials are given by $[n]_q^!=\prod_{i=1}^n[i]_q$ with $[i]_q=\frac{q^i-1}{q-1}$. 
  Then, let us compute $\sum_{U\subset V}\gamma_{U}$ which is the sum, over all possible tuples $(a_S)_{[S]\in\mathcal{S}}$, of the number of sub-representations (\ref{eq:lemma}) times the coefficient $\gamma_U$ of the statement of the Lemma: 
  \[\sum_{U\subset V}\gamma_{U}=\sum_{a_S\leq m_S}  \prod_{[S]\in\mathcal{S}}\genfrac{[}{]}{0pt}{}{m_S}{a_S}_{|\End(S)| }(-1)^{a_S}|\End(S)|^{\binom{a_S}{2}}=\prod_{[S]\in\mathcal{S}}\sum_{a=0}^{m_S} \genfrac{[}{]}{0pt}{}{m_S}{a}_{|\End(S)| }(-1)^{a}|\End(S)|^{\binom{a}{2}}\]
Showing (by induction) that 
$\sum_{a=0}^{m} \genfrac{[}{]}{0pt}{}{m}{a}_{q }(-1)^{a}q^{\binom{a}{2}}=0$ if $m\geq 1$, we obtain that $\sum_{U\subset V}\gamma_{U}=0$, completing the proof. 
\end{proof}

Then we can compute $I(z^{-1})$:
\[I(z^{-1})=I\big( \sum_{[V]}\gamma_V [V]\big) = \sum_{[V]}\frac{\gamma_V}{|\Aut V|} t^{\dim V}=\]
\[\sum_{(m_S)_{S\in\mathcal{S}}} \prod_{[S]\in\mathcal{S}} \frac{(-1)^{m_S} |\End S|^{\binom{m_S}{2}}}{|\Aut \bigoplus_{[S]\in \mathcal{S}} S^{m_S}|} t^{\sum_{[S]\in \mathcal{S}}m_S \dim S}=\]
\[\sum_{(m_S)_{S\in\mathcal{S}}} \prod_{[S]\in\mathcal{S}} \left(\frac{(-1)^{m_S} |\End S|^{\binom{m_S}{2}}}{|\GL_{m_S}(\End S)|} t^{\sum_{[S]\in \mathcal{S}}m_S \dim S}\right)=\]
\begin{equation}\label{eq:aux1}
    \prod_{[S]\in\mathcal{S}} \left(\sum_{m\geq 0}\frac{(-1)^{m} |\End S|^{\binom{m}{2}}}{|\GL_{m}(\End S)|} t^{\sum_{[S]\in \mathcal{S}}m \dim S}\right)\; .
\end{equation}
Observe that, for each irreducible representation $S\in\mathcal{S}$, we have $|\End S|=q^b$, where $b=\dim_{\mathbb{F}_q}\End S$. Then, (\ref{eq:aux1}) is equal to 
\begin{equation}\label{eq:aux2}
    \prod_{[S]\in\mathcal{S}} \left(\sum_{m\geq 0}\frac{(-1)^{m} (q^b)^{\binom{m}{2}}}{|\GL_{m}(\mathbb{F}_{q^b})|} t^{\sum_{[S]\in \mathcal{S}}m \dim S}\right)\; .
\end{equation}
By applying (\ref{eq:GLnFq}) to the finite field of $q^b$ elements we obtain that (\ref{eq:aux2}) equals
\begin{equation}\label{eq:aux3}
    \prod_{[S]\in\mathcal{S}} \left(\sum_{m\geq 0}\left(\prod_{i=1}^m (1-(q^b)^i)^{-1}\right) t^{\sum_{[S]\in \mathcal{S}}m \dim S}\right)\; .
\end{equation}
Defining $s_{\alpha, b}(q)$ the number of isomorphism classes of irreducible representations $S$ of $\mathbb{F}_q F_r$ with dimension $\alpha$ and $\dim_{\mathbb{F}_q}\End S =b$, we get that (\ref{eq:aux3}) is equal to 
\begin{equation}\label{eq:aux4}
\prod_{\alpha\in\mathbb{N}^I, b\geq 1}\left( \sum_{m\geq 0} \prod_{i=1}^m (1-q^{bi})^{-1}\right)^{s_{\alpha, b}(q)} t^{m\alpha}=\prod_{\alpha\in\mathbb{N}^I, b\geq 1}\PExp\left(\frac{t^{\alpha}}{1-q^{b}}\right)^{s_{\alpha, b}(q)} \; .
\end{equation}
To finish, we need to make use of certain arithmetic functions (see \cite[Lemma 2.3, Corollary 3.3, Theorem 4.2]{mozgovoy2009number}) to conclude that 
\begin{equation}\label{eq:aux5}
I(z^{-1})=\PExp\left(\frac{1}{1-q}\sum_{n\geq 1} A_{n,r}^{\ast}(q) t^n\right)\; .
\end{equation}
Relating with (\ref{eq:Iz-1}), we finally obtain a relationship for the generating series whose coefficients give the number of irreducible representations of the free group in $r$ generators over $\mathbb{F}_q$.
\begin{Proposition}\cite[Theorem 2.5]{mozgovoy2015arithmetic}\label{prop:generating}
The generating series of irreducible representations of the free group in $r$ generators over $\mathbb{F}_q$ can be computed from the plethystic operators as
    \begin{equation}\label{eq:series_Airr}
   \sum_{n\geq 1} A_{n,r}^{\ast}(q) t^n= (1-q)\PLog\left[S\left(F(t)^{-1}\right)\right] \; .
\end{equation}
\end{Proposition}

\begin{Exercise}
    Search for material in the literature to complete the details to get expressions (\ref{eq:aux4}) and (\ref{eq:aux5}), towards the generating series of the number of irreducible representations of the free group into $\GL_n$, for $n\geq 1$. 
\end{Exercise}

\begin{Remark}\label{rem:Katz_over_Airr}
In \cite[Theorem 2.5]{mozgovoy2015arithmetic} it is further shown that 
\begin{equation}\label{eq:generating_series}
\sum_{n\geq 0} A_{n,r}(q) t^n=\PExp\left( \sum_{n\geq 1} A_{n,r}^{\ast}(q) t^n\right)\; , 
\end{equation}
which relates the generating series of irreducible representations to the generating series of all representations. Moreover, the coefficients $A_{n,r}(q)$ and $A_{n,r}^{\ast}(q)$ are in fact polynomials in $q$ with integer coefficients. By Katz's result Theorem \ref{thm:Katz}, these are the $e$-polynomials of the $\GL_n$-character varieties of the free group, which are balanced varieties and their polynomials are $1$-variable polynomials. 

In \cite[Theorem 4.10]{florentino2023generating}, this result is generalized to $\GL_n$-character varieties of an arbitrary finitely generated group $\Gamma$, whose character varieties are not necessarily of balanced type, then their $e$-polynomials are, in principle, $2$-variable polynomials. 
\end{Remark}

\begin{Exercise}
Use \cite[Lemma 5]{mozgovoy2007computational} and  \cite[Theorem 2.5]{mozgovoy2015arithmetic} to prove the relationship in Remark \ref{rem:Katz_over_Airr}, 
\[
\sum_{n\geq 0} A_{n,r}(q) t^n=\PExp\left( \sum_{n\geq 1} A_{n,r}^{\ast}(q) t^n\right)\; ,
\]
relating the generating series of irreducible representations to the generating series of all representations. Then we have the generating series for the coefficients $A_{n,r}(q)$:
\[
   \sum_{n\geq 1} A_{n,r}(q) t^n= \PExp\left[(1-q)\PLog[S(F(t)^{-1})] \right]\; .
\]
\end{Exercise}

\subsection{Stratifications of \texorpdfstring{$\GL_n$}{GLn}-character varieties by partition type}\label{ssec:stratification}

The seminal work \cite{logares2013hodge} computes the mixed Hodge polynomials and $e$-polynomials of $\SL_2$-character varieties of the fundamental group of a Riemann surface of small genus, $g=1,2$. The method used is purely geometrical, opposed to the arithmetic ideas described in the previous subsection, and relies on decomposing the character variety into Jordan types of matrices, given that the conjugacy action defining the character varieties respect this decomposition. For each Jordan type, the $e$-polynomial of that strata is obtained by a careful study of the fibrations appearing there and the monodromy action (c.f. Theorem \ref{thm:fibrations}). The $e$-polynomial of the whole character variety results as the sum of each piece, by Proposition \ref{prop:e-strat}. This strategy has produced many results, for example \cite{martinez2016thesis} for general genus $g$ and \cite{lawton2016SL3} for $\SL_3$. 

In this section we are going to describe a similar but slightly different technique to stratify $\GL_n$-character varieties into Jordan or semi-simplicity types, which we will call partition types. This is described in \cite{florentino2023generating}. 

\begin{Definition}\label{def:partitions}
Denote by $\mathcal{P}_n$ the \textbf{partitions} of the natural number $n$, whose elements are  
\begin{equation}\label{eq:partition}
    [k]:=[1^{k_{1}} 2^{k_{2}} \cdots n^{k_{n}}]\in \mathcal{P}_{n}\; ,
\end{equation}
where $\sum_{j=1}^{n} j\cdot k_{j} = n$. Let $|[k]|=\sum_{j=1}^{n} k_{j}$ be its \textbf{length}, the number of elements counted with its multiplicity. 
\end{Definition}

Recall that $\mathcal{R}_{\GL_n}(\Gamma)$ denotes the set of representations $\rho:\Gamma\rightarrow \GL_n$.

\begin{Definition}\label{def:kpolystable}
Denote by $\mathcal{R}_{\GL_n}^{[k]}(\Gamma)$
the \textbf{$[k]$-polystable representations} which are those representations $\rho$ conjugated to a direct sum $\bigoplus_{j=1}^{n} \rho_{j}$, where each summand $\rho_{j}\in \mathcal{R}_{\GL_j^{\oplus k_{j}}}^{\ast}(\Gamma)$ is an irreducible representation of the corresponding size.
\end{Definition}

Observe that 
\[\mathcal{R}_{\GL_n}(\Gamma)=\bigsqcup_{[k]\in\mathcal{P}_n} \mathcal{R}^{[k]}_{\GL_n}(\Gamma)\; .\] 
Since the stabilizer dimension of a representation is invariant by conjugation, the action of $\GL_n$ on $\mathcal{R}_{\GL_n}(\Gamma)$ respects each $\mathcal{R}_{\GL_n}^{[k]}(\Gamma)$, then we can define 
the \textbf{$[k]$-stratum} as 
\begin{equation}\label{eq:kstratum}
    \mathcal{X}_{\GL_n}^{[k]}(\Gamma) := \mathcal{R}_{\GL_n}^{[k]}(\Gamma)/\!\!/\GL_{n}\; .
\end{equation}
Observe that the \textbf{irreducible} stratum of irreducible representations corresponds to the trivial partition $[n]$, i.e. $\mathcal{X}_{\GL_n}^{[n]}(\Gamma)=\mathcal{X}_{\GL_n}^{\ast}(\Gamma)$. 

\begin{Theorem}\cite[Proposition 4.3]{florentino2023generating}\label{thm:stratGLn}
There exists a locally closed stratification 
\[\mathcal{X}_{\GL_{n}}(\Gamma) = \bigsqcup_{[k] \in \mathcal{P}_{n}} \mathcal{X}_{\GL_n}^{[k]}(\Gamma)\; .\]
\end{Theorem}
\begin{proof}
  Observe that the stable locus $\mathcal{R}_{\GL_n}^{[n]}(\Gamma)$, i.e. the irreducible representations, is an open subset of $\mathcal{R}_{\GL_n}(\Gamma)$ and  $\mathcal{R}_{\GL_n}(\Gamma)/\GL_n$ is a geometric quotient inside the GIT quotient. Then, its complement, the reducible representations, is a closed subset. Repeating the argument for partitions of greater length we obtain a locally closed stratification where each $[k]$-strata is an irreducible component of the representations sharing a stabilizer of the same dimension.   
\end{proof}

Thanks to Theorem \ref{thm:stratGLn}, the computation of the $e$-polynomial of a $\GL_n$-character variety reduces to the computation for each strata, by Proposition \ref{prop:e-strat}. But this still forces us to understand the monodromy action of GIT quotients by $\GL_n$, which is an infinite group. Let us further reduce this computations to understand the monodromy under finite quotients.

\begin{Definition}\label{def:kLevi}
Let $[k]\in \mathcal{P}_n$ be a partition and let $\mathcal{X}_{\GL_n}^{[k]}(\Gamma)$ be the corresponding stratum of the $\GL_n$-character variety of $\Gamma$.
Define the \textbf{$[k]$-Levi} by 
\begin{equation}\label{eq:kLevi}
    L_{[k]}:=\GL_{1}^{k_{1}}\times \GL_{2}^{k_{2}}\times \cdots \times \GL_{n}^{k_{n}}\subset \GL_{n}\; ,
\end{equation}
the \textbf{$[k]$-symmetric group}
\begin{equation}\label{eq:kS}
    S_{[k]}:=S_{k_{1}}\times S_{k_{2}}\times \cdots \times S_{k_{n}}\subset S_{n}\; ,
\end{equation} and the \textbf{$[k]$-normalizer} 
\begin{equation}\label{eq:kN}
    N_{[k]}=L_{[k]}\rtimes S_{[k]}\; .
\end{equation}  
\end{Definition} 
The conjugation action of $\GL_n$ under the stratum $\mathcal{X}_{\GL_n}^{[k]}(\Gamma)$ reduces to the action of the $[k]$-normalizer (\ref{eq:kN}), then 
notice that
\begin{equation}\label{eq:K-strata}
\mathcal{X}_{\GL_n}^{[k]}(\Gamma)\simeq \mathcal{R}_{\GL_n}^{[k]}(\Gamma)/\!\!/N_{[k]} =\mathcal{R}_{\GL_n}^{[k]}(\Gamma)/\!\!/\big(L_{[k]}\rtimes S_{[k]}\big) \simeq \prod_{j=1}^n\left(\mathcal{X}_{\GL_j}^{\ast}(\Gamma)^{\times k_j}\right)/S_{k_j}\; ,
\end{equation}
where the last isomorphism comes from including the action of each block in the $[k]$-Levi (\ref{eq:kLevi}), into the corresponding irreducible representation to give an irreducible character variety. Therefore,  we are describing each stratum as a product of irreducible character varieties acted by finite symmetric groups. 

\begin{Proposition}\label{prop:decomp_epol}
The $e$-polynomial of the $\GL_n$-character variety of $\Gamma$ is computed additively on the strata as
\[e\left(\mathcal{X}_{\GL_{n}}(\Gamma)\right) = \sum_{[k] \in \mathcal{P}_{n}} e\left(\mathcal{X}_{\GL_n}^{[k]}(\Gamma)\right)=\sum_{[k] \in \mathcal{P}_{n}} e\left(\prod_{j=1}^n\left(\mathcal{X}_{\GL_j}^{\ast}(\Gamma)^{\times k_j}\right)/S_{k_j}\right)\; .\]  
\end{Proposition}
\begin{proof}
Using the stratification in Theorem \ref{thm:stratGLn} and the additivity of the $e$-polynomial under locally closed subsets in Proposition \ref{prop:e-strat}, we can compute the $e$-polynomial of the $\GL_n$-character variety of $\Gamma$ as
\begin{eqnarray*}
e\left(\mathcal{X}_{\GL_{n}}(\Gamma)\right) = \sum_{[k] \in \mathcal{P}_{n}} e\left(\mathcal{X}_{\GL_n}^{[k]}(\Gamma)\right)=\sum_{[k] \in \mathcal{P}_{n}} e\left(\mathcal{R}_{\GL_n}^{[k]}(\Gamma)/\!\!/N_{[k]} \right)=\\
\sum_{[k] \in \mathcal{P}_{n}} e\left(\mathcal{R}_{\GL_n}^{[k]}(\Gamma)/\!\!/\big(L_{[k]}\rtimes S_{[k]}\big)\right)=
\sum_{[k] \in \mathcal{P}_{n}} e\left(\prod_{j=1}^n\left(\mathcal{X}_{\GL_j}^{\ast}(\Gamma)^{\times k_j}\right)/S_{k_j}\right)\; . 
\end{eqnarray*}    
\end{proof}

\begin{Remark}
Note that, in the expression of Proposition \ref{prop:decomp_epol} we cannot compute each $e$-polynomial in the sum by partition types, just by multiplying the $e$-polynomials of each irreducible character variety in the product. This is due to the fact that the action of $S_{[k]}$ is not diagonal in the product and permutes representations in the $k_j$ irreducible $\GL_j$-character varieties of the same $j$. 
\end{Remark}

We conclude the section with an example of a stratification of a character variety. 

\begin{Example}\label{ex:GL10}
Let us describe one of the strata in the stratification of the $\GL_{10}$-character variety of a general group $\Gamma$.

Consider the partition $[k]=[1^3\; 2^2\;3]\in \mathcal{P}_{10}$, where $1\cdot 3 + 2\cdot 2+3=10$. The 
\textbf{$[k]$-polystable representations} in $\mathcal{R}_{\GL_{10}}^{[k]}(\Gamma)$ are those representations $\rho$ such that the image of all generators of $\Gamma$ are simultaneously conjugated to a matrix of this shape:     
\[\left(\begin{array}{cccccccccc|}
        \multicolumn{1}{|c}{\ast} & \ast & \ast & \ast & \ast & \ast & \ast & \ast & \ast & \ast \\\cline{1-1}
         \multicolumn{1}{c|}{} & \ast & \ast & \ast & \ast & \ast & \ast & \ast & \ast & \ast \\\cline{2-2}
          & \multicolumn{1}{c|}{}  & \ast & \ast & \ast & \ast & \ast & \ast & \ast & \ast \\\cline{3-3}
         & & \multicolumn{1}{c|}{} & \ast & \ast & \ast & \ast & \ast & \ast & \ast \\
         & & \multicolumn{1}{c|}{} & \ast & \ast & \ast & \ast & \ast & \ast & \ast \\\cline{4-5}
         & & & & \multicolumn{1}{c|}{}& \ast & \ast & \ast & \ast & \ast \\
         & & & & \multicolumn{1}{c|}{}& \ast & \ast & \ast & \ast & \ast \\\cline{6-7}
         & & & & &  &  \multicolumn{1}{c|}{} & \ast & \ast & \ast \\
         & & & & &  &  \multicolumn{1}{c|}{} & \ast & \ast & \ast \\
         & & & & &  &  \multicolumn{1}{c|}{} & \ast & \ast & \ast \\\cline{8-10}
    \end{array}\right)\]
The associated \textbf{$[k]$-Levi} are matrices of this form:
\[L_{[k]} =  \left\{\left(\begin{array}{cccccccccc}\cline{1-1}
        \multicolumn{1}{|c|}{\ast}  & & & & & & & & & \\\cline{1-2}
         \multicolumn{1}{c|}{} & \multicolumn{1}{c|}{\ast}  &  &  &  &  &  &  &  & \\\cline{2-3}
          & \multicolumn{1}{c|}{}  & \multicolumn{1}{c|}{\ast} &  &  &  &  &  &  &  \\\cline{3-5}
         & & \multicolumn{1}{c|}{} &  \ast & \multicolumn{1}{c|}{\ast}  &  &  &  &  \\
         & & \multicolumn{1}{c|}{} & \ast & \multicolumn{1}{c|}{\ast} &  &  &  &  &  \\\cline{4-7}
         & & & & \multicolumn{1}{c|}{}& \ast & \multicolumn{1}{c|}{\ast} &  &  &  \\
         & & & & \multicolumn{1}{c|}{}& \ast & \multicolumn{1}{c|}{\ast} &  &  &  \\\cline{6-10}
         & & & & &  &  \multicolumn{1}{c|}{} & \ast & \ast & \multicolumn{1}{c|}{\ast} \\
         & & & & &  &  \multicolumn{1}{c|}{} & \ast & \ast & \multicolumn{1}{c|}{\ast} \\
         & & & & &  &  \multicolumn{1}{c|}{} & \ast & \ast & \multicolumn{1}{c|}{\ast} \\\cline{8-10}
    \end{array}\right)\right\}=\GL_1^{\times 3}\times \GL_2^{\times 2}\times \GL_3\; .\]
And the \textbf{$[k]$-symmetric group} is $S_{[k]}=S_3\times S_2$, permuting the three blocks of size one and the two blocks of size two. 

Therefore, the $e$-polynomial of the stratum can be computed as
\[e\left(\mathcal{X}_{\GL_{10}}^{[k]}(\Gamma)\right)=e\left(\mathcal{R}_{\GL_{10}}^{[k]}(\Gamma)/\!\!/\big(L_{[k]}\rtimes S_{[k]}\big)\right)=
 e\left(\mathcal{X}_{\GL_1}^{\ast}(\Gamma)^{\times 3}/S_{3}\times \mathcal{X}_{\GL_2}^{\ast}(\Gamma)^{\times 2}/S_{2}\times \mathcal{X}_{\GL_3}^{\ast}(\Gamma)\right)\; .\]
\end{Example}

\section{An explicit computation: \texorpdfstring{$e$}{e}-polynomial of the \texorpdfstring{$\GL_3$}{GL3}--character variety of the free group}
\label{ssec:explicitAirr}

In this section we use the techniques presented before to actually compute something: the $e$-polynomial of the $\GL_3$-character variety of the free group. We follow the arithmetic methods of \cite{mozgovoy2015arithmetic} combined with the geometric technique of partitions in \cite{florentino2023generating, florentino2021serre}.

\subsection{Step 1: \texorpdfstring{$e$}{e}-polynomials of irreducible character varieties of the free group}

Recall (c.f. \cite[Appendix]{hausel2008mixed}) that the $e$-polynomials of $\mathcal{X}^{\ast}_{\GL_n}(F_r)$, the irreducible $\GL_n$-character varieties over the free group,  equal the polynomials 
$A^{\ast}_{n,r}(q)$, counting the number of points of these varieties over the finite field $\mathbb{F}_q$. Let us use the combinatorial techniques in \cite{mozgovoy2009number, mozgovoy2015arithmetic, florentino2023generating, florentino2021serre} to compute these polynomials.

\begin{Proposition}\cite[Proposition 6.2]{florentino2021serre}
\label{prop:Airrformula} 
The $e$-polynomials of the
irreducible character varieties $A^{\ast}_{n,r}(x)=e(\mathcal{X}^{\ast}_{\GL_n}(F_r))$
are explicitly given by: 
\[
A^{\ast}_{n,r}(x)=(x-1)\sum_{m|n}\frac{\mu(n/m)}{n/m}\,\sum_{[k]\in\mathcal{P}_{m}}\frac{(-1)^{|[k]|}}{|[k]|}\binom{|[k]|}{k_{1},\cdots,k_{m}}\prod_{j=1}^{m}b_{j}(x^{n/m})^{k_{j}}x^{\frac{n(r-1)k_{j}}{m}\binom{j}{2}}\;,
\]
where the polynomials $b_{j}(x)$ are the coefficients of the series $F^{-1}(t)=1+\sum_{n\geq1}b_{n}(x)t^{n}$ and $F(t)$ is given in (\ref{eq:F}). 
\end{Proposition}
\begin{proof}
By (\ref{eq:series_Airr}) and (\ref{eq:pletyhistic_Adams}), we have that generating series of the polynomials $A^{\ast}_{n,r}(x)$ satisfy
\[
\sum_{n\geq1}A^{\ast}_{n,r}(x)\,t^{n}=(1-x)\PLog\left[S \left(F(t)^{-1}\right)\right]=(1-x)\Psi^{-1}\left[\log\left(S (F(t)^{-1})\right)\right]\;,
\]
with $F(t)$ as in \eqref{eq:F}, and $S$ as in (\ref{eq:shift}), then 
\[
S \left(F(t)^{-1}\right)=1+\sum_{n\geq1}b_{n}(x)\,x^{(r-1)\binom{n}{2}}\,t^{n}\;.
\]
Using the Taylor series $\log(1+z)=z-\frac{z^{2}}{2}+\frac{z^{3}}{3}-\cdots\;$,
and the multinomial theorem: 
\[\log\left(1+\sum_{n\geq1}b_{n}(x)\,x^{(r-1)\binom{n}{2}}\,t^{n}\right)=\]
 \[\left(\sum_{n\geq1}b_{n}(x)\,x^{(r-1)\binom{n}{2}}\,t^{n}\right)-\frac{1}{2}\left(\sum_{n\geq1}b_{n}(x)\,x^{(r-1)\binom{n}{2}}\,t^{n}\right)^{2}+\frac{1}{3}\left(\sum_{n\geq1}b_{n}(x)\,x^{(r-1)\binom{n}{2}}\,t^{n}\right)^{3}-\cdots=\]
\begin{equation}\label{eq:log_series}\sum_{m\geq1}\left[\sum_{[k]\in\mathcal{P}_{m}}\frac{(-1)^{|[k]|-1}}{|[k]|}\binom{|[k]|}{k_{1},\cdots,k_{m}}\prod_{j=1}^{m}b_{j}(x)^{k_{j}}\,x^{k_j(r-1)\binom{j}{2}}\right]t^{m}\; ,
\end{equation}
where 
\[\binom{|[k]|}{k_{1},\cdots,k_{m}}=\frac{|[k]|!}{k_{1}!\cdots k_{m}!}\]
are the multinomial coefficients and the inner sum runs over partitions $\mathcal{P}_m$ from Definition \ref{def:partitions}. 

Finally, let us apply the $\mathbb{Q}$-linear operator  $\Psi^{-1}$ in (\ref{eq:Adams-1}) to the expression (\ref{eq:log_series}):
\begin{eqnarray*}
\Psi^{-1}\left(\sum_{m\geq1}\left[\sum_{[k]\in\mathcal{P}_{m}}\frac{(-1)^{|[k]|-1}}{|[k]|}\binom{|[k]|}{k_{1},\cdots,k_{m}}\prod_{j=1}^{m}b_{j}(x)^{k_{j}}\,x^{k_j(r-1)\binom{j}{2}}\right]t^{m}\right)=\\
\sum_{m\geq1}\Psi^{-1}\left(\left[\sum_{[k]\in\mathcal{P}_{m}}\frac{(-1)^{|[k]|-1}}{|[k]|}\binom{|[k]|}{k_{1},\cdots,k_{m}}\prod_{j=1}^{m}b_{j}(x)^{k_{j}}\,x^{k_j(r-1)\binom{j}{2}}\right]t^{m}\right)=\\
\sum_{m\geq 1} \sum_{[k]\in\mathcal{P}_m}\frac{(-1)^{|[k]|-1}}{|[k]|}\binom{|[k]|}{k_{1},\cdots,k_{m}} \prod_{j=1}^m \sum_{l\geq 1}\frac{\mu(l)}{l}b_j(x^l)^{k_j}x^{l k_j (r-1)\binom{j}{2}}t^{lm}=\\
\sum_{n\geq 1} \sum_{m | n} \frac{\mu(n/m)}{n/m} \sum_{[k]\in\mathcal{P}_m}\frac{(-1)^{|[k]|-1}}{|[k]|}\binom{|[k]|}{k_{1},\cdots,k_{m}} \prod_{j=1}^m b_j(x^{n/m})^{k_j}x^{\frac{n k_j (r-1)}{m}\binom{j}{2}}t^{n}
\end{eqnarray*}
The proposition follows from here, multiplying by $(x-1)$. 
\end{proof}

Recall that 
\[F(t)=1+\sum_{n\geq1}a_{n}(x)\,t^{n}\quad\text{and }\quad F^{-1}(t)=1+\sum_{n\geq1}b_{n}(x)t^{n}\; ,
\]
where $a_{n}(x):=\big((x-1)(x^{2}-1)\ldots(x^{n}-1)\big)^{r-1}$, are formal inverses. They are related by $\sum_{k\geq0}a_{k}b_{n-k}=0$, $a_{0}=b_{0}=1$.  From this, we can compute the first three $b_n(x)$ recursively as 
\begin{equation}\label{eq:bn}
b_{1}=-a_{1}\;,\quad b_{2}=a_{1}^{2}-a_{2}\;,\quad b_{3}=-a_{1}^{3}+2a_{1}a_{2}-a_{3}\; ,
\end{equation}
obtaining 
\begin{align*}
b_{1}(x) & =-a_{1}(x)=-(x-1)^{r-1}\;,\\
b_{2}(x) & =a_{1}^{2}(x)-a_{2}(x)=(x-1)^{2r-2}-(x-1)^{r-1}(x^{2}-1)^{r-1}\\
 & =(x-1)^{2r-2}\big(-(x+1)^{r-1}+1\big)\;,\\
b_{3}(x) & =-a_{1}^{3}(x)+2a_{1}(x)\,a_{2}(x)-a_{3}(x)=\\
 & =(x-1)^{3r-3}\big(-(x+1)^{r-1}(x^{2}+x+1)^{r-1}+2(x+1)^{r-1}-1\big)\; .
\end{align*}
Then, by substituting in the formula of Proposition \ref{prop:Airrformula}, we get explicit calculations of the $e$-polynomials of the irreducible character
varieties $A_{n,r}^{\ast}(x)=e(\mathcal{X}_{\GL_n}^{\ast}(F_r))$, for $n=1,2,3$.

\begin{Proposition}\label{prop:calculationAirr}
The $e$-polynomials of the irreducible character
varieties $A_{n,r}^{\ast}(x)=e(\mathcal{X}_{\GL_n}^{\ast}(F_r))$, for $n=1,2,3$, are given by 
  \begin{eqnarray*}\label{eq:Airr}
e(\mathcal{X}_{\GL_1}^{\ast}(F_r))=A_{1,r}^{\ast}(x) & = & (x-1)^{r}\;,\\
e(\mathcal{X}_{\GL_2}^{\ast}(F_r))=A_{2,r}^{\ast}(x)& = & (x-1)\left(\frac{1}{2}b_{1}(x^{2})+\frac{1}{2}b_{1}(x)^{2}-b_{2}(x)x^{r-1}\right)\\
 & = & (x-1)^{r}\Big((x-1)^{r-1}x^{r-1}((x+1)^{r-1}-1)+\frac{1}{2}(x-1)^{r-1}-\frac{1}{2}(x+1)^{r-1}\Big)\; ,\\
e(\mathcal{X}_{\GL_3}^{\ast}(F_r))=A_{3,r}^{\ast}(x)& = & (x-1)\left(\frac{1}{3}b_{1}(x^{3})-\frac{1}{3}b_{1}(x)^{3}+b_{1}(x)b_{2}(x)x^{r-1}-b_{3}(x)x^{3r-3}\right)\\
 & = & (x-1)^{r}\Big(-\frac{1}{3}(x^{2}+x+1)^{r-1}+(x-1)^{2r-2}(\frac{1}{3}-x^{r-1}+x^{r-1}(x+1)^{r-1}\\
 &  & +x^{3r-3}+x^{3r-3}(x+1)^{r-1}(x^{2}+x+1)^{r-1}-2x^{3r-3}(x+1)^{r-1})\Big)\; .
\end{eqnarray*}  
\end{Proposition}

\begin{Exercise}
    Compute $b_4(x)$ as in (\ref{eq:bn}) and compute $e(\mathcal{X}_{\GL_4}^{\ast}(F_r))=A_{4,r}^{\ast}(x)$ as in Proposition \ref{prop:calculationAirr}.
\end{Exercise}

\subsection{Step 2: Computation of the abelian strata}

Let us compute separately the $e$-polynomial of the smallest and most singular strata, $\mathcal{X}^{[1^n]}_{\GL_n}(F_r)$, which corresponds to representations of the free group in $r$ generators into $\GL_n$ which are simultaneously reducible to diagonal matrices for each generator. This stratum will be called the \textbf{abelian stratum}, due to the fact that their representations reduce to an abelian subgroup of $\GL_n^r$. The following lemma shows that studying representations in the partition $[1^n]$ is equivalent to studying representations of the abelianization of $F_r$, this is of $\mathbb{Z}^r$, into $\GL_n$.

\begin{Lemma}\cite[Lemma 5.2]{florentino2023generating}
\label{lem:abelian} 
The abelian
stratum is isomorphic to the character variety of the abelianization
of $F_r$: 
\[
\mathcal{X}^{[1^{n}]}_{\GL_n}(F_r)\cong\mathcal{X}_{\GL_n}(\mathbb{Z}^r)\; .
\]
\end{Lemma}
\begin{proof}
By Schur lemma, if $\phi\in \Aut (\mathbb{C}^n)$ and $\rho:\Gamma\rightarrow \GL_n$ is an irreducible representation such that $\rho(\gamma)\in \GL_n$ commutes with $\phi$ for every $\gamma\in\Gamma$, then $\phi$ must be a scalar multiple of the identity. Then, if $\Gamma=\mathbb{Z}^r$ abelian and $\rho$ is irreducible, $\rho(\gamma)$ is a multiple of the identity for every $\gamma\in\Gamma$, hence 
the irreducible representations of the abelian group $\mathbb{Z}^r$ are necessarily $1$-dimensional.

Therefore, a polystable representation $\rho:\mathbb{Z}^r\rightarrow \GL_n$ splits into a direct sum of irreducible representations and, hence, belongs to the strata $\mathcal{R}^{[1^n]}_{\GL_n}(\mathbb{Z}^r)$. By composing with the quotient $F_r\twoheadrightarrow \mathbb{Z}^r$ we get a representation of $\mathcal{R}^{[1^n]}_{\GL_n}(F^r)$, then $\mathcal{R}^{ps}_{\GL_n}(\mathbb{Z}^r)\subset \mathcal{R}^{[1^n]}_{\GL_n}(F^r)$, this inclusion of the polystable points being a morphism of algebraic varieties.

On the other hand, a representation of $\mathcal{R}^{[1^n]}_{\GL_n}(F^r)$ has its image into an abelian subgroup of $\GL_n^r$, then under the quotient $F_r\twoheadrightarrow \mathbb{Z}^r$ it defines a unique representation of $\mathbb{Z}^r$ given that all commutators are sent to the identity in $\GL_n$. 
Then, we obtain an isomorphism 
\[
\mathcal{R}^{[1^{n}]}_{\GL_n}(F_r)\cong\mathcal{R}^{ps}_{\GL_n}(\mathbb{Z}^r)\; ,
\]
which provides the isomorphism of the statement of the Lemma by taking quotients to the character varieties. 
\end{proof}

Now let us compute the $e$-polynomial of the abelian character variety $\mathcal{X}_{\GL_n}(\mathbb{Z}^r)$. For a detailed study on abelian character varieties, see \cite{florentino2021abelian, li2024SPSO}. 

\begin{Proposition}\cite[Theorem 5.1]{florentino2021abelian}
\label{prop:abelianZn} 
The $e$-polynomial of the $\GL_n$-character variety of the free abelian group in $r$ generators is 
\[
e\left((\mathcal{X}_{\GL_n}(\mathbb{Z}^{r})\right)=\sum_{[k]\in\mathcal{P}_{n}}\prod_{j=1}^{n}\frac{(x^j-1)^{r\,k_j}}{k_j!\, j^{k_j}}\; .
\]
\end{Proposition}
\begin{proof}
As we said in the proof of Lemma \ref{lem:abelian}, all irreducible representations of the abelian group $\mathbb{Z}^r$ into $\GL_n$ are $1$-dimensional, therefore 
\[e(\mathcal{X}_{\GL_n}^{\ast}(\mathbb{Z}_r))=A_{n,r}^{\ast, \mathbb{Z}_r}(x)=0\; , \quad \text{for }n\geq 2.\]
For $n=1$, the conjugation action in the $\GL_1=\mathbb{C}^{\ast}$-character variety is trivial, then 
\[e(\mathcal{X}_{\GL_1}(\mathbb{Z}^r))=e(\mathcal{R}_{\GL_1}(\mathbb{Z}^r))=e(\Hom(\mathbb{Z}^r, \mathbb{C}^{\ast}))=e((\mathbb{C}^{\ast})^{r})=(x-1)^r\; ,\]
using the computations in (\ref{eq:eC*}) and Proposition \ref{prop:relation_pols}. 

Now, using the relationship (\ref{eq:generating_series}) between the generating series, we obtain:
\[\sum_{n\geq 0} A_{n,r}^{\mathbb{Z}_r}(x) t^n=\PExp\left(A_{1,r}^{\ast, \mathbb{Z}_r}(x) t\right)=\PExp\left((x-1)^r t\right)\; .\]

Let us compute this plethystic exponential using its definition in (\ref{eq:PExp}): 
\[
\PExp((x-1)^r t)=e^{\left(\Psi((x-1)^rt)\right)}=e^{\left(\sum_{j\geq1}\frac{(x^j -1)^r t^j}{j}\right)}=
\prod_{j\geq1}e^{\left(\frac{(x^j -1)^r t^j}{j}\right)}=\]
\[\prod_{j\geq1}\,\sum_{k\geq0}\frac{(x^j-1)^{rk} t^{jk}}{k!\,j^{k}}=
\sum_{n\geq0}t^{n}\left(\sum_{[k]\in\mathcal{P}_{n}}\prod_{j=1}^{n}\frac{(x^j-1))^{r k_{j}}}{k_{j}!\ j^{k_{j}}}\right)\; ,
\]
where note that, in the last equality, we put together all terms contributing to $t^n$, which are precisely all partitions $n=\sum_{j=1}^n j k_j$. The proposition follows from this. 
\end{proof}

Let us finish the section by applying these results. 

\begin{Proposition}\label{prop:calculation1^3}
The $e$-polynomial of the abelian stratum of the $\GL_3$-character variety of the free group in $r$ generators is 
\[e\left(\mathcal{X}^{[1^{3}]}_{\GL_3}(F_r)\right)=(x-1)^r\left(\frac{(x^2+x+1)^r}{3}
+\frac{(x^2-1)^r}{2}+\frac{(x-1)^{2r}}{6}\right)\; .\]
\end{Proposition}
\begin{proof}
By Lemma \ref{lem:abelian} and Proposition \ref{prop:abelianZn}, we have 
\[e\left(\mathcal{X}^{[1^{3}]}_{\GL_3}(F_r)\right)=e\left((\mathcal{X}_{\GL_3}(\mathbb{Z}^{r})\right)=\sum_{[k]\in\mathcal{P}_{3}}\prod_{j=1}^{3}\frac{(x^{j}-1)^{r\,k_{j}}}{k_{j}!\ j^{k_{j}}}\; .\]
There are three possible partitions in $\mathcal{P}_3$ of the integer $3$:
\begin{eqnarray*}
 &[3]\; , & \text{with } k_1=0, k_2=0, k_3=1\; ,\\
 &[1 2]\; , & \text{with } k_1=1, k_2=1, k_3=0\; ,\\
& [13]\; , & \text{with } k_1=3, k_2=0, k_3=0\; .\
\end{eqnarray*}
Then, by plugging these numbers into the previous expression, we get:
\begin{eqnarray*}
\sum_{[k]\in\mathcal{P}_{3}}\prod_{j=1}^{3}\frac{(x^{j}-1)^{r\,k_{j}}}{k_{j}!\ j^{k_{j}}} &  & =\\
  \frac{(x-1)^{r\cdot 0}}{0!\; 1^0}\cdot\frac{(x^2-1)^{r\cdot 0}}{0!\; 2^0}\cdot \frac{(x^3-1)^{r\cdot 1}}{1!\; 3^1} & + & \\
    \frac{(x-1)^{r\cdot 1}}{1!\; 1^1}\cdot\frac{(x^2-1)^{r\cdot 1}}{1!\; 2^1}\cdot \frac{(x^3-1)^{r\cdot 0}}{0!\; 3^0} & + & \\
      \frac{(x-1)^{r\cdot 3}}{3!\; 1^3}\cdot\frac{(x^2-1)^{r\cdot 0}}{0!\; 2^0}\cdot \frac{(x^3-1)^{r\cdot 0}}{0!\; 3^0} &  & =\\
      \frac{(x^3-1)^r}{3}
+(x-1)^r\frac{(x^2-1)^r}{2}+\frac{(x-1)^{3r}}{6}\; ,
\end{eqnarray*}
from which we get the expression of the statement, factoring by $(x-1)^r$.
\end{proof}

\begin{Exercise}
Compute $e\left(\mathcal{X}^{[1^{4}]}_{\GL_4}(F_r)\right)$ (c.f. Proposition \ref{prop:calculation1^3}).
\end{Exercise}

\subsection{Step 3: \texorpdfstring{$e$}{e}-polynomial of the \texorpdfstring{$\GL_3$}{GL3}--character variety of the free group}

Finally, let us compute explicitly the $e$-polynomial of $\mathcal{X}_{\GL_3}(F_r)$, the $\GL_3$-character variety of the free group in $r$ generators. 

By Theorem \ref{thm:stratGLn}, we will stratify the character variety into three strata, corresponding to the 
three possible partitions of the integer $3$, 
$\mathcal{P}_3=\{[3], [1\; 2], [1^3]\}$: 
\begin{equation}
    \mathcal{X}_{\GL_{3}}(F_r) = \bigsqcup_{[k] \in \mathcal{P}_{3}} \mathcal{X}_{\GL_3}^{[k]}(F_r)=\mathcal{X}^{[3]}_{\GL_{3}}(F_r)\sqcup \mathcal{X}^{[1\; 2]}_{\GL_{3}}(F_r)\sqcup \mathcal{X}^{[1^3]}_{\GL_{3}}(F_r)\; .
\end{equation}
The stratum labeled by the partition $[3]$ corresponds to the irreducible representations and the label $[1^3]$ corresponds to the abelian stratum. 
By (\ref{prop:e-strat}), the $e$-polynomial can be obtained as the sum of the $e$-polynomials of each stratum:
\begin{equation}
e(\mathcal{X}_{\GL_{3}}(F_r)) = e(\mathcal{X}^{[3]}_{\GL_{3}}(F_r))+ e(\mathcal{X}^{[1\; 2]}_{\GL_{3}}(F_r))+e(\mathcal{X}^{[1^3]}_{\GL_{3}}(F_r))\; .
\end{equation}
Using Proposition \ref{prop:decomp_epol} we can further decompose the calculation of the first two strata into those for irreducible character varieties and get
\begin{eqnarray*}
e(\mathcal{X}_{\GL_3}^{\ast}(F_r))+e(\mathcal{X}_{\GL_1}^{\ast}(F_r))\cdot e(\mathcal{X}_{\GL_2}^{\ast}(F_r))+e(\mathcal{X}^{[1^3]}_{\GL_{3}}(F_r))=\\
A_{3,r}^{\ast}(x)+A_{2,r}^{\ast}(x)\cdot A_{1,r}^{\ast}(x)+e(\mathcal{X}^{[1^3]}_{\GL_{3}}(F_r))\; .
\end{eqnarray*}
The $e$-polynomials $A_{j,r}^{\ast}, j=1,2,3$ of the three irreducible character varieties are computed in Proposition \ref{prop:calculationAirr} and the $e$-polynomial of the abelian character variety $\mathcal{X}^{[1^3]}_{\GL_{3}}(F_r)$ is computed in Proposition \ref{prop:calculation1^3}. 

Gathering together all the computations we obtain the following

\begin{Theorem}\cite[Theorem 6.7]{florentino2023generating}\label{thm:computationGL3}
The $e$-polynomial of the $\GL_3$-character variety of the free group in $r$ generators is
\[
 e\left(\mathcal{X}_{\GL_3}(F_r)\right)  =  (x-1)^r\left[
\frac{1}{3}(x^2+x+1)^{r-1}(x+1)x+\frac{1}{2}(x-1)^r(x+1)^{r-1}x\right.\]
\[+\left.(x-1)^{2r-2}\left((x+1)^{r-1}\left(x^{3r-3}(x^2+x+1)^{r-1} + x^r-2x^{3r-3}\right)+x^{3r-3}-x^r+\frac{1}{6}x(x-1)\right)\right] \; .\]
\end{Theorem}

We can obtain some direct geometric information from the $e$-polynomial computed in Theorem \ref{thm:computationGL3}. 

\begin{Corollary}\label{cor:euler_irr}
The $\GL_3$-character variety of the free group in $r$ generators is an irreducible affine algebraic variety of complex dimension $9r-8$ and Euler characteristic zero.
\end{Corollary}
\begin{proof}
    Note that the leading monomial of $ e\left(\mathcal{X}_{\GL_3}(F_r)\right)$ has exponent (see the first terms of the second line) 
    \[r+(2r-2)+(r-1)+(3r-3)+2(r-1)=9r-8\; .\]
    The leading coefficient is equal to one, then the top Hodge number is one, hence the irreducibility (c.f. Remark \ref{rem:irreducible}). By setting $x=1$, the factor $(x-1)^r$ yields the Euler characteristic equal to zero. 
\end{proof}

\begin{Remark}
    To proceed with the calculation of  $e\left(\mathcal{X}_{\GL_4}(F_r)\right)$ in rank $4$ is not automatic from the steps followed in Theorem \ref{thm:computationGL3}. This is due to the presence of partition $[2^2]$ whose polynomial is not the product of the polynomials of two irreducible character varieties, but a quotient of these by $S_2$. See the notion of rectangular partitions in \cite{florentino2023generating, florentino2021serre} for details on achieving the calculation for general $n$. 
\end{Remark}

\section{Stratifications of \texorpdfstring{$G$}{G}--character varieties}

In this section and the following one we sketch the results of \cite{gonzalez2024rootdata} where the authors extend the $\GL_n$-stratification in \cite{florentino2023generating} (see Section \ref{ssec:stratification}) to a general $G$-character variety, where $G$ is a complex reductive algebraic group. We will produce a locally closed stratification of $\mathcal{X}_G(\Gamma)$ into strata indexed by the parabolic subgroups of $G$, and will reduce the computation of each strata to a subvariety of representations (forming a core and a pseudo-quotient), which simplifies considerably the problem. 

\subsection{Pseudo-quotients and cores}

 Here we recall the main ideas of \cite{gonzalez2024pseudo} about the notions of pseudo-quotient and core of an action of a group in a variety. We will observe that pseudo-quotients are the right quotient notion to deal with motives in the Grothendieck ring (c.f. Remark \ref{rem:motives}) and, therefore, with $e$-polynomials. 

In the following, we denote by $(X,G)$ the pair consisting on an algebraic variety $X$ and a complex reductive algebraic group acting $G$ acting on $X$. 

\begin{Definition}\label{def:quotients}
Let $(X,G)$ be a pair as before and let $(Y,\pi)$ be another pair consisting on an algebraic variety $Y$ and a regular morphism $\pi:X\rightarrow Y$ such that
\begin{itemize}
\item[(1)] $\pi$ is surjective.
\item[(2)] $\pi$ is $G$-invariant.
\item[(3)] $W\subseteq X$ closed $G$-invariant, then $\pi(W)\subseteq Y$ closed. 
\item[(4)] If $W_1, W_2\subseteq X$ closed $G$-invariant, $W_1\cap W_2=\emptyset \Leftrightarrow \pi(W_1)\cap \pi(W_2)=\emptyset$. 
\item[(5)] If $V\subseteq Y$ open, $\pi$ induces an isomorphism $\pi^{\ast}:\mathcal{O}_{Y}(V)\cong \mathcal{O}_{X}(\pi^{-1}(V))^{G}\subseteq \mathcal{O}_{X}(\pi^{-1}(V))$.
\item[(6)] It is an orbit space, i.e. $G\cdot x$ is closed in $X$ for every $x\in X$. 
\end{itemize}

If (1), (2), (3) and (4) are satisfied, then $(Y,\pi)$ is a \textbf{pseudo-quotient} of $(X,G)$.

If, moreover, (5) is satisfied, then $(Y,\pi)$ is a \textbf{good quotient} of $(X,G)$.

If, furthermore, (6) is satisfied, then $(Y,\pi)$ is a \textbf{geometric quotient} of $(X,G)$
\end{Definition}

The notion of good quotient is the one used in Geometric Invariant Theory \cite{mumford1994geometric, newstead2012introduction}, and therefore character varieties are good quotients, which are unique and categorical. However, good quotients do not behave well motivically, in the sense that they may not commute with stratifications. Pseudo-quotients, by contrast, are not unique nor categorical quotients. This is a weaker notion capturing topology but not algebra in the sense that if $\pi$ is a pseudo-quotient, for every open subset $V\subseteq Y$, $\pi$ induces the morphism
\[\pi^{\ast}:\mathcal{O}_{Y}(V)\rightarrow \mathcal{O}_{X}(\pi^{-1}(V))^{G}\hookrightarrow \mathcal{O}_{X}(\pi^{-1}(V))\]
which is not necessarily an isomorphism. However, they capture completely the motivic information of the quotient. 

\begin{Proposition}\cite[Proposition 3.7, Corollary 3.8, Corollary 4.3]{gonzalez2024pseudo}
If $(Y, \pi)$, $(Z, \xi)$ are two pseudo-quotients of $(X,G)$, then the motives $[Y]=[Z]$ are equal in ${\rm KVar}$ and, then, the $e$-polynomials $e(Y)=e(Z)$ are equal. 
\end{Proposition}

Using the notion of pseudo-quotient, we will reduce the motivic computation of each stratum to a core, which comes from the following idea. 
Given a pair $(X,G)$, suppose that there exists a subvariety $Y \subseteq X$ such that it meets all closures of orbits, this is, $\overline{G\cdot x}\cap Y\neq \emptyset$ for every $x\in X$. Then, for every point in the GIT quotient $[x]\in X/\!\!/G$, there exists a representative $y\in Y$ with $[x]=[y]$. However, $Y$ is not a slicing: we must quotient by a subgroup $H \subseteq G$ leaving $Y$ invariant. This idea resembles on the notion of polystable points in moduli theory: all $S$-equivalence classes in a moduli space contain a polystable representative. 

\begin{figure}[h!]
    \centering
    \includegraphics[clip, trim=4cm 4cm 1cm 5cm, width=1.1\linewidth]{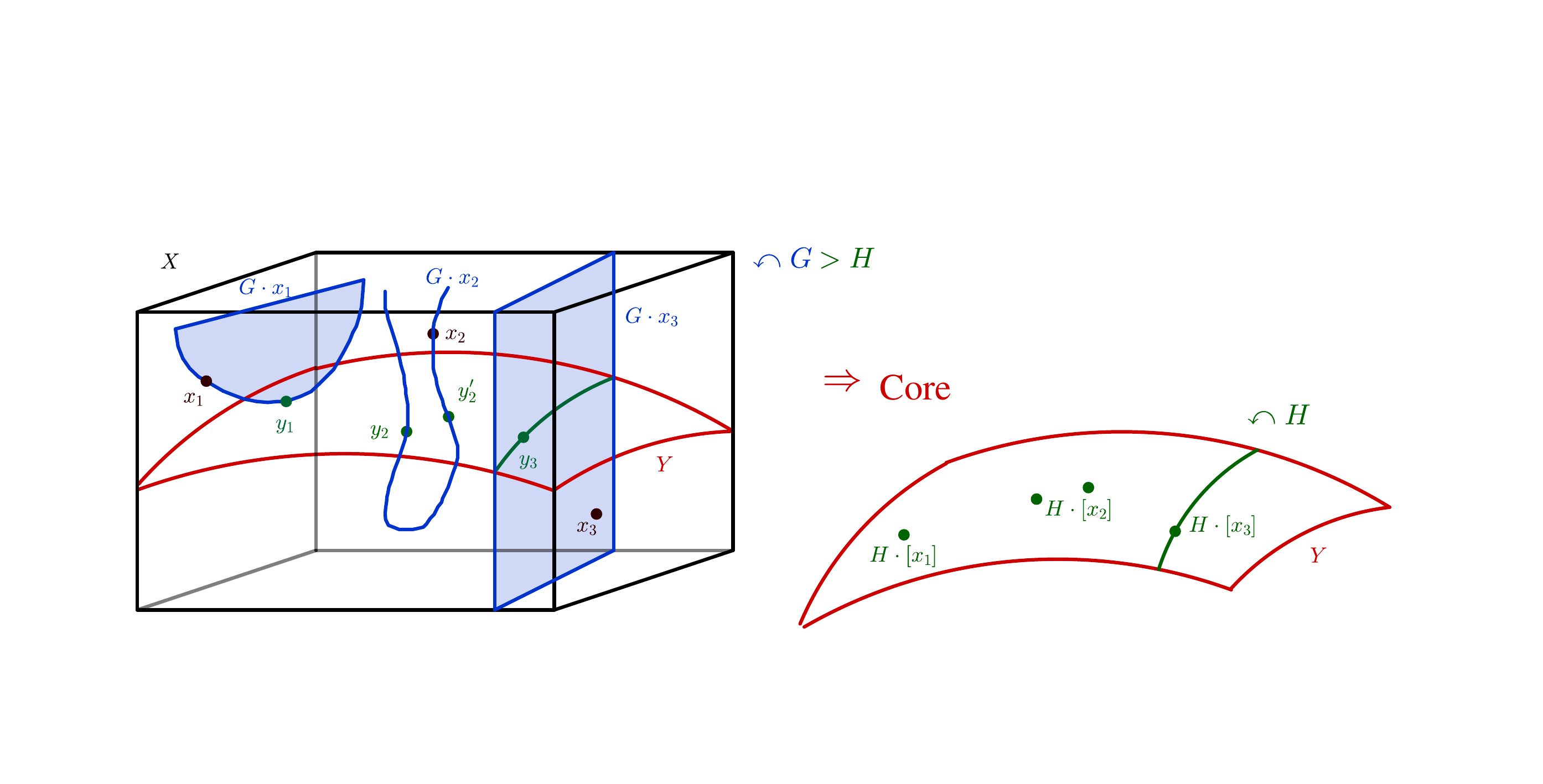}
    \caption{Core $(Y,H)$ of $(X,G)$.}
    \label{fig:core}
\end{figure}

\begin{Definition}\label{def:core}
A \textbf{core} for $(X,G)$ is a pair $(Y, H)$ given by a subvariety $Y\subseteq X$ and a subgroup $H \subseteq G$ such that 
\begin{itemize}
\item[(i)] $Y$ is orbitwise-closed for the $H$-action, i.e. $\overline{H\cdot y}\subseteq Y$, for all $y\in Y$. 
\item[(ii)] For every $x\in X$, we have $\overline{G\cdot x}\cap Y\neq \emptyset$. 
\item[(iii)] For every two $W_1, W_2\subseteq Y$ disjoint closed (in $Y$) $H$-invariant subsets, we have that $\overline{G\cdot W_1}\cap \overline{G\cdot W_2} = \emptyset$. 
\end{itemize}
\end{Definition}

\begin{Proposition}\cite[Proposition 5.8]{gonzalez2024pseudo}
Suppose that $(X,G)$ has a core $(Y, H)$ and $(X,G)$ admits a categorical quotient. Then, for any two pseudo-quotients $X\rightarrow \overline{X}$, $Y \rightarrow \overline{Y}$ of $(X,G)$ and $(Y,H)$ respectively, we have
$[\overline{X}]=[\overline{Y}]$ in ${\rm KVar}$.

In particular, if there exists GIT good quotients $X\rightarrow X/\!\!/ G$ and $Y\rightarrow Y/\!\!/H$, then 
the motives $[X/\!\!/ G]=[Y/\!\!/H]$ are equal in ${\rm KVar}$ and, therefore, the $e$-polynomials $e\left(X/\!\!/ G\right)=e\left(Y/\!\!/H\right)$ are equal. 
\end{Proposition}

\subsection{Root data}

Here we collect some basic knowledge about root data of Lie groups. For an extensive treatment on this, see \cite{Borel, springer1998}. 

Let $G$ be a complex reductive algebraic group. 
Fix a \textbf{Borel} subgroup $B$ and a \textbf{maximal torus} $T$ such that $T\subset B\subset G$. Each conjugacy class of \textbf{parabolics} $\mathcal{P}$ contains a unique standard parabolic $P$ such that $T\subset B\subset P\subset G$.

The \textbf{characters} and \textbf{cocharacters} of $T$ are dual free abelian groups
\[\chi\in X^{\ast}(T)=\{\chi:T\rightarrow \mathbb{C}^{\ast}\}\quad , \quad \lambda \in X_{\ast}(T)=\{\lambda:\mathbb{C}^{\ast}\rightarrow T\}\; ,\]
with pairing $(\chi, \lambda )\in \mathbb{Z}$. 
The \textbf{roots} $\Phi\subseteq X^{\ast}(T)$ are weights for the adjoint action of $T$ on ${\rm Lie}(G)$. Roots are in bijection with the \textbf{coroots} $\Phi^{\vee}\subseteq X_{\ast}(T)$ by $\Phi \ni \alpha\leftrightarrow \alpha^{\vee}\in \Phi^{\vee}$ satisfying the pairing $(\alpha, \alpha^{\vee})=2$. 

The \textbf{Weyl group} of $G$ is defined as $W\simeq N_G(T)/Z_G(T) \subset {\rm Aut}(X^{\ast}(T))$, and it is generated by the \textbf{reflections} 
\[s_{\alpha}:X^{\ast}(T)\rightarrow X^{\ast}(T)\;,\; x\mapsto x-(x, \alpha^{\vee})\alpha\; .\]

We will call \textbf{root datum} of $G$ to the set 
\[R = (X^{\ast}, \Phi, X_{\ast}, \Phi^{\vee})\]
given by characters, roots, cocharacters and coroots. The information in the root datum recovers completely the group $G$. The \textbf{dual root datum} is given by interchanging characters with cocharacters, and roots with coroots:
\[R^\vee = (X_{\ast}, \Phi^{\vee}, X^{\ast}, \Phi)\; .\]
This dual root datum recovers another reductive group which is called the \textbf{Langlands dual} group of $G$ and it is denoted by $^L G$.
For example, the complex groups 
$G=\SL_n:=\SL(n,\mathbb{C})$ and $^L G=\PGL_n:=\PGL(n,\mathbb{C})$ are Langlands dual groups, as well as $G=\Sp_{2n}:=\Sp(2n,\mathbb{C})$ and $^L G=\SO_{2n+1}:=\SO(2n+1,\mathbb{C})$. The groups 
$\GL_n:=\GL(n,\mathbb{C})$ and $\SO_{2n}:=\SO(2n,\mathbb{C})$ are Langlands \textbf{self-dual}.

Let us recall how parabolic subgroups are given by subsets of the Dynkin diagram. 
Pick a set of positive roots $\Phi^{+}$ and a set of \textbf{simple roots} $\Delta\subset \Phi^{+}$ (which are the nodes of Dynkin diagram of the semisimple part of $G$). 
Subsets $I\subseteq \Delta$ are in bijection with conjugacy classes of parabolics $\mathcal{P}_I$ or standard parabolics such that $T\subset B\subset P_I\subset G$. Observe that $P_{\emptyset}=B$ and $P_{\Delta}=G$. There are $2^{|\Delta|}$ (conjugacy classes of) parabolics.

For each $I\subset \Delta$ we define define the following groups: 
\begin{itemize}
\item $\Phi_I\subseteq \Phi$ is the subset of roots generated by $\alpha\in I$, and we define $\Phi^{\vee}_I$ similarly, 
\item \textbf{$I$-torus}, defined as $T_I:=\left(\bigcap_{\alpha\in I}\ker \alpha\right)^{\circ}\subseteq T$, 
\item \textbf{$I$-Levi subgroup}, defined as $L_I:=Z_G(T_I)$, 
\item \textbf{$I$-normalizer}, defined as $N_I:=N_G(T_I)$, 
\item \textbf{$I$-Weyl}, defined as $W_I:=N_G(T_I)/Z_G(T_I)=N_I/L_I$ (not to be confused with the Weyl group of $L_I$).
\end{itemize}
Observe that the root datum of $L_I$ is $(X^{\ast}(T), \Phi_I, X_{\ast}(T), \Phi^{\vee}_I)$.

\subsection{Reducing motivic computations via parabolic stratification and cores}

Let $\mathcal{X}_G(\Gamma)$ be the $G$-character variety of a finitely generated group $\Gamma$. Let us produce a stratification of $\mathcal{X}_G(\Gamma)$ into parabolic types in order to compute their motives and $e$-polynomials  using the idea of pseudo-quotient and core. 

First, we decompose the representation space $\mathcal{R}_G(\Gamma)=\bigcup_{I\subseteq \Delta}\widehat{\mathcal{R}}^{\ast}_{P_I}(\Gamma)$ into \textbf{conjugacy classes of parabolic representations} where we will denote: 
\begin{itemize}
\item $\mathcal{R}_{P_I}(\Gamma):=\{\rho:\Gamma\rightarrow G\;,\; \rho(\Gamma)\subset P_I\}$, the 
\textbf{$P_I$-representations},
\item $\widehat{\mathcal{R}}_{P_I}(\Gamma):=\bigcup_{P\in\mathcal{P}_I}\mathcal{R}_P(\Gamma)=G\cdot \mathcal{R}_{P_I}(\Gamma)$, the \textbf{conjugacy class of $P_I$-representations},
\item $\mathcal{R}_{P_I}^{\ast}(\Gamma)= \mathcal{R}_{P_I}(\Gamma)\backslash \bigcup_{J\subsetneq I}\mathcal{R}_{P_J}(\Gamma)$, the \textbf{$P_I$-irreducible representations},
\item $\widehat{\mathcal{R}}^{\ast}_{P_I}(\Gamma)=G\cdot \mathcal{R}_{P_I}^{\ast}(\Gamma)$, the \textbf{conjugacy class of $P_I$-irreducible representations}.
\end{itemize}
And define similarly the subsets $\mathcal{R}_{L_I}(\Gamma)$, $\widehat{\mathcal{R}}_{L_I}(\Gamma)$, $\mathcal{R}^{\ast}_{L_I}(\Gamma)$, $\widehat{\mathcal{R}}^{\ast}_{L_I}(\Gamma)$ for \textbf{Levi $L_I$-representations}.

The following result produces the desired motivic decomposition via parabolic stratification. 
\begin{Theorem}\cite[Theorem 4.13 and corollary 4.14]{gonzalez2024rootdata}\label{thm:GZ_strat}
For each subset $I\subset \Delta$, the pair $(\mathcal{R}_{L_I}^\star, N_I)$ is a core for $(\widehat{\mathcal{R}}_{P_I}^\star, G)$. Therefore the motives 
\[ [\widehat{\mathcal{R}}_{P_I}^\star(\Gamma)/\!\!/ G] = [\mathcal{R}_{L_I}^\star(\Gamma)/\!\!/ N_I]\] 
are equal in ${\rm KVar}$ and, hence, the $e$-polynomials 
\[ e\left(\widehat{\mathcal{R}}_{P_I}^\star(\Gamma)/\!\!/ G\right) = e\left(\mathcal{R}_{L_I}^\star(\Gamma)/\!\!/ N_I\right)\] 
are also equal. 
\end{Theorem}
\begin{proof}
To prove this result we check the conditions in Definition \ref{def:core} of core. First we show that $\mathcal{R}_{L_I}^\star$ is polystable and $N_I$-invariant, therefore it is orbitwise-closed. 
Then, if a representation $\rho\in \widehat{\mathcal{R}}^{\ast}_{P_I}(\Gamma)$, then $\overline{G\cdot \rho}\cap \mathcal{R}^{\ast}_{L_I}(\Gamma)\neq \emptyset$. Finally, if two representations $\rho_1, \rho_2 \in \mathcal{R}^{\ast}_{L_I}(\Gamma)$ satisfy $\overline{G\cdot \rho_1}\cap \overline{G\cdot \rho_2}\neq\emptyset$, then $\exists g_0\in N_I$ such that $g_0\cdot \rho_1=\rho_2$.
\end{proof}

\begin{Theorem}\cite[Theorem 4.15 and Corollary 4.17]{gonzalez2024rootdata}\label{thm:moticivGZ}
For every reductive group $G$ and every finitely generated group $\Gamma$, we have that 
\[\left[\mathcal{X}_G(\Gamma)\right] = 
 \left[\mathcal{R}_G(\Gamma)/\!\!/ G \right]= \left[\bigcup_{I \subseteq \Delta} \widehat{\mathcal{R}}_{P_I}^\star(\Gamma)/\!\!/ G\right]=
\sum_{I \in 2^{\Delta}/\sim_W} \left[\mathcal{R}_{L_I}^\star(\Gamma)/\!\!/ N_I\right]\; ,\] 
and also 
\[e\left(\mathcal{X}_G(\Gamma)\right) = 
 e\left(\mathcal{R}_G(\Gamma)/\!\!/ G \right)= e\left(\bigcup_{I \subseteq \Delta} \widehat{\mathcal{R}}_{P_I}^\star(\Gamma)/\!\!/ G\right)=
\sum_{I \in 2^{\Delta}/\sim_W} e\left(\mathcal{R}_{L_I}^\star(\Gamma)/\!\!/ N_I\right)\; .\] 

If, moreover, the sequence
\[    1 \longrightarrow L_I = Z_G(T_I) \longrightarrow N_I = N_G(T_I) \longrightarrow W_I = N_I / L_I \longrightarrow 1\]
splits, hence, $N_I = L_I \rtimes W_I$, then the decomposition simplifies to
\[\left[\mathcal{X}_G(\Gamma)\right] = \sum_{I \in 2^{\Delta}/\sim_W} \left[\mathcal{X}_{L_I}^\star(\Gamma)/\!\!/ W_I\right]\quad \text{and} \quad e\left(\mathcal{X}_G(\Gamma)\right) = \sum_{I \in 2^{\Delta}/\sim_W} e\left(\mathcal{X}_{L_I}^\star(\Gamma)/\!\!/ W_I\right)\; .\]
\end{Theorem}
\begin{proof}
We provide a sketch of the proof. We have that $\mathcal{R}_G(\Gamma)=\bigcup_{I\subseteq \Delta}\widehat{\mathcal{R}}^{\ast}_{P_I}(\Gamma)$, then  show that $\widehat{\mathcal{R}}^{\ast}_{P_I}(\Gamma)$ are $G$-invariant, orbitwise-closed and locally-closed.
By Theorem \ref{thm:GZ_strat}, we have $[\widehat{\mathcal{R}}_{P_I}^\star(\Gamma)/\!\!/ G] = [\mathcal{R}_{L_I}^\star(\Gamma)/\!\!/ N_I]$. 
Now, note that $\widehat{\mathcal{R}}^{\ast}_{P_I}(\Gamma)$ are not disjoint. However, if $\rho \in \mathcal{R}_{L_I}^\ast(\Gamma)$ is conjugated to $\rho' \in \mathcal{R}_{L_I'}^\star(\Gamma)$ for $I \neq I'$, then $L_I$ is conjugated to $L_{I'}$ by $w\in W$, and therefore $I \sim_W I'$.

For the second statement, notice that if the sequence splits, then
\[
\mathcal{R}_{L_I}^\star(\Gamma) \sslash N_I = \mathcal{R}_{L_I}^\star(\Gamma) \sslash (L_I \rtimes W_I) =  \left(\mathcal{R}_{L_I}^\star(\Gamma) \sslash L_I\right) \sslash W_I = \mathcal{X}_{L_I}^\star(\Gamma) \sslash W_I\; .
\]
\end{proof}

\section{Motivic computations for ABCD Lie groups}

In this final section we describe the parabolic stratification of the $G$-character variety, for certain groups $G$ of Dynkin type A, B, C and D. The reader con consult the details in \cite[Section 5]{gonzalez2024rootdata}. 

\subsection{Stratification for \texorpdfstring{$\GL_n$}{GLn}-, \texorpdfstring{$\PGL_n$}{PGLn}- and \texorpdfstring{$\SL_n$}{SLn}-, type A}

This case recovers the results in \cite{florentino2023generating}, by re-interpreting them in terms of root data in Lie theory. 

Let us consider first the reductive group $\GL_n$ whose Dynkin diagram  is $\dynkin[labels={1,2,n-2,n-1},label directions={,,,,,},scale=2]A{}$ of type $A_{n-1}$. 
In this case, simple roots are labeled as $\Delta = \{1, \ldots, n-1\}$. 

We fix a basis $(e_1, e_2,\ldots, e_n)$ of $\mathbb{C}^n$ and choose $B$ upper triangular invertible matrices.
Then, the subsets $I \subseteq \Delta\;,\;\Delta\setminus I=\{i_1,i_2,\ldots,i_s\}$ are in bijection with standard parabolic subgroups $P_I$ of the form 
\[    \left(\begin{array}{ccccc|}
        \multicolumn{1}{|c}{\ast} & \ast & \ast & \ast & \ast \\\cline{1-1}
         & \multicolumn{1}{|c}{\ast} & \ast & \ast & \ast \\
         & \multicolumn{1}{|c}{\ast} & \ast & \ast & \ast \\\cline{2-3}
         &  &  & \multicolumn{1}{|c}{\ast} & \ast \\
         &  &  & \multicolumn{1}{|c}{\ast} & \ast \\\cline{4-5}
    \end{array}\right)\]
Parabolics are, in turn, stabilizers of the flag
\[ 0 \subsetneq V_{1} \subsetneq V_{2} \subsetneq \cdots \subsetneq V_{s} \subsetneq \mathbb{C}^n\; ,\] 
where $V_j=\langle e_1,\ldots, e_{i_j} \rangle$. In particular: 
\begin{itemize}
\item $B$ corresponds to $I=\emptyset\;,\;\Delta\setminus I=\Delta=\{1,2,\ldots,n-1\}$ and to the full flag 
\[0\subsetneq V_1\subsetneq V_2\subsetneq \cdots \subsetneq V_{n-1}\subsetneq \mathbb{C}^n\; .\]
\item $\GL_n$ corresponds to $I=\Delta\;,\;\Delta\setminus I=\emptyset$ and to the trivial flag $0\subsetneq \mathbb{C}^n$. 
\item Maximal parabolics $P_I$ correspond to complements of single nodes $I=\Delta\setminus\{i_1\}\;,\;\Delta\setminus I=\{i_i\}$  and to $1$-step flags $0\subsetneq V_{1}\subsetneq \mathbb{C}^n$. 
\end{itemize}
Levi subgroups are $L_I \cong \prod_{j = 1}^{s+1} \GL_{\lambda_j}$, corresponding to matrices of the form 
\[    \left(\begin{array}{ccccc}\cline{1-1}
        \multicolumn{1}{|c|}{\ast} & \multicolumn{4}{r}{} \\\cline{1-3}
         & \multicolumn{1}{|c}{\ast} & \multicolumn{1}{c|}{\ast} & \multicolumn{2}{r}{} \\
         & \multicolumn{1}{|c}{\ast} & \multicolumn{1}{c|}{\ast} & \multicolumn{2}{r}{} \\\cline{2-5}
        \multicolumn{1}{r}{}  &  &  & \multicolumn{1}{|c}{\ast} & \multicolumn{1}{c|}{\ast}\\
        \multicolumn{1}{r}{}  &  &  & \multicolumn{1}{|c}{\ast} & \multicolumn{1}{c|}{\ast}\\\cline{4-5}
    \end{array}\right)\]
which are stabilizers of the graded object of the filtration given by $P_I$, this is $\mathbb{C}^n = \bigoplus_{j = 1}^{s+1} V_j/V_{j-1}$, with $\lambda_j=\dim V_j/V_{j-1}$. In particular:
\begin{itemize}
\item $L_{\emptyset}=T_{\GL_n}$ is the maximal torus.
\item $L_{\Delta}=\GL_n$ is the whole group.
\end{itemize}

Recall (Definition \ref{def:partitions}) the notion of partitions $[k]=[1^{k_{1}}\cdots j^{k_{j}}\cdots n^{k_{n}}]\in \mathcal{P}_{n}$ and notice that these are in bijection with subsets $I / \sim_W \; \in 2^\Delta / \sim_W$, under the action of the Weyl group (permuting blocks of the same size). Observe that the $I$-normalizers are 
\[N_I=N_{\GL_n}(T_I) = N_{[k]}=L_{[k]} \rtimes S_{[k]}\; ,\] therefore, this completely recovers Theorem \ref{thm:stratGLn}, isomorphisms (\ref{eq:K-strata}) and Proposition \ref{prop:e-strat} from \cite{florentino2023generating}.

Let us describe the stratifications for the Langlands dual groups $G=\SL_n$ and $^L G=\PGL_n$ of type $A_{n-1}$, whose Dynkin diagram is also $\dynkin[labels={1,2,n-2,n-1},label directions={,,,,,},scale=2]A{}$. 

The stratification for $\SL_n$ is 
\[\mathcal{X}_{\SL_n}(\Gamma)=\bigsqcup_{[k]\in\mathcal{P}_{n}}\mathcal{X}^{[k]}_{\SL_n}(\Gamma)\; ,\]
where a representation $\rho$ belongs to the $[k]$-stratum if 
\[\rho = \bigoplus_{j=1}^{n}\bigoplus_{\ell = 1}^{k_j}\rho_{j,\ell} \in \mathcal{R}_{L_{[k]}^{\SL_n}}^{\ast}(\Gamma)\; .\] 
Here each $\rho_{j,\ell} \in \mathcal{R}^{\ast}_{\GL_{j}}(\Gamma)$ is irreducible with $\prod_{j,\ell}\det(\rho_{j, \ell}) = 1$.
Associated Levi subgroups are 
\[L_{[k]}^{\SL_n}  = \left\{(A_{j,\ell}) \in \prod_{j=1}^n \GL_{j}^{k_j} \,\left|\, \prod_{j,\ell} \det(A_{j,\ell}) = 1\right.\right\} \subseteq \prod_{j=1}^n \GL_{j}^{k_j}\; ,\]
and each stratum is described as
\[\mathcal{X}^{[k]}_{\SL_n}(\Gamma):=\mathcal{X}_{L_{[k]}^{\SL_n}}^{\ast}(\Gamma)  /\!\!/ S_{[k]} = \left\{(\rho_{j,\ell}) \in \prod_{j=1}^n \mathcal{X}^{\ast}_{\GL_j}(\Gamma)^{k_j} \,\left|\, \prod_{j,\ell} \det(\rho_{j,\ell}) = 1\right.\right\}/\!\!/ S_{[k]}.\]

The stratification for $\PGL_n$ is 
\[\mathcal{X}_{\PGL_n}(\Gamma)=\bigsqcup_{[k]\in\mathcal{P}_{n}}\mathcal{X}^{[k]}_{\PGL_n}(\Gamma)\; ,\]
where a representation $\rho$ belongs to the $[k]$-stratum if 
\[\rho= \left(\bigoplus_{j=1}^{n}\bigoplus_{\ell = 1}^{k_j}\rho_{j,\ell}\right)/\mathbb{C}^{\ast}\in \mathcal{R}_{L_{[k]}^{\PGL_n}}^{\ast}(\Gamma)\; .\] 
Now, each $\rho_{j,\ell} \in \mathcal{R}^{\ast}_{\GL_{j}}(\Gamma)$ is irreducible.
Associated Levi subgroups are 
\[L_{[k]}^{\PGL_n} = \left\{(A_{j,\ell}) \in \prod_{j=1}^n \GL_{j}^{k_j}\right\}/\mathbb{C}^{\ast}\; ,\]
and each stratum is described as
\[\mathcal{X}^{[k]}_{\PGL_n}(\Gamma):=\mathcal{X}_{L_{[k]}^{\PGL_n}}^{\ast}(\Gamma) /\!\!/ S_{[k]} = \left(\left(\prod_{j=1}^n \mathcal{X}^{\ast}_{\GL_j}(\Gamma)^{k_j}\right) / \mathbb{C}^{\ast} \right) /\!\!/ S_{[k]}\; .\]
Topological mirror symmetry conjectures that certain topological invariants should be equal, or mirror in some sense, for geometrical objects constructed out of Langlands dual groups $G$ and $^L G$. In particular, it is conjectured that the motives and $e$-polynomials of the character varieties $\mathcal{X}_G(\Gamma)$ and $\mathcal{X}_{^L G}(\Gamma)$ are equal. This has been proven to be true in some cases, \cite{Hausel_Thaddeus_2003} for $\SL_2, \PGL_2$ and the fundamental group of a Riemann surface, \cite{groechenig2020} for $\SL_n, \PGL_n$ and the fundamental group of a Riemann surface, and \cite{florentino2021serre} for $\SL_n, \PGL_n$ and the free group. Other cases remain as an open problem to the best of our knowledge. 

One of the applications of the stratification of the $G$-character variety into parabolic types is that it is motivic, then checking mirror symmetry boils down to check it strata by strata. 

\begin{Theorem}\cite{florentino2021serre, gonzalez2024rootdata}\label{thm:TMS_SL_PGL}
Topological mirror symmetry holds $[\mathcal{X}_{\SL_n}(\Gamma)] = [\mathcal{X}_{\PGL_n}(\Gamma)]$ (resp.  $e(\mathcal{X}_{\SL_n}(\Gamma)) = e(\mathcal{X}_{\PGL_n}(\Gamma))$) if and only if $[\mathcal{X}^{[k]}_{\SL_n}(\Gamma)] = [\mathcal{X}^{[k]}_{\PGL_n}(\Gamma)]$ (resp. $e(\mathcal{X}^{[k]}_{\SL_n}(\Gamma)) = e(\mathcal{X}^{[k]}_{\PGL_n}(\Gamma))$) holds strata by strata.
\end{Theorem}

\subsection{\texorpdfstring{$\Sp_{2n}$}{Sp(2n)}, type C and \texorpdfstring{$\SO_{2n+1}$}{SO(2n+1)}, type B}

Let us begin by considering $G=\Sp_{2n}$ whose Dynkin diagram is $\dynkin[labels={1,2,n-2,n-1,n},label directions={,,,,,},scale=2]C{}$ of type $C_{n}$. Here, simple roots are $\Delta=\{1,\ldots,n\}$ with $n$ the unique long root. 

We fix a basis $(e_1,\ldots,e_n,f_1,\ldots,f_n)$ of $\mathbb{C}^{2n}$ such that $\omega=\sum_{i=1}^n e_i^*\wedge f_i^*$ is the standard symplectic form on $\mathbb{C}^{2n}$. Given vectors $u,v\in\mathbb{C}^{2n}$, with coordinates $(x^u_1, \ldots, x^u_n, y^u_1, \ldots, y^u_n)$ and $(x^v_1, \ldots, x^v_n, y^v_1, \ldots, y^v_n)$ respect to the basis, we have $\omega(u, v)=\sum_{i=1}^n x_i^u y_i^v - y_i^u x_i^v$. 

Subsets $I \subseteq \Delta$, $\Delta \setminus I = \{i_1, \ldots, i_s\}$ are in bijection with standard parabolic subgroups $P_I$ which are stabilizers of the flag 
\[0 \subsetneq V_{1} \subsetneq \cdots \subsetneq V_{j} \subsetneq \cdots \subsetneq V_{s} \subsetneq V_{s}^\perp \subsetneq V_{{s-1}}^\perp \subsetneq \cdots \subsetneq V_{1}^\perp \subsetneq \mathbb{C}^{2n}\; ,\]  
where $V_{j} = \langle x_{1}, \ldots, x_{i_j}\rangle$ for $1 \leq j \leq s$ are \textbf{isotropic subspaces} (under the symplectic form) and $V_j^\perp = \langle x_1, \ldots, x_n, y_{i_j+1}, \ldots, y_n\rangle$.
There is a special maximal parabolic for $I=\Delta\setminus\{n\}$ yielding a \textbf{lagrangian} (maximal isotropic) flag 
\[0\subsetneq V_1=\langle x_1,\ldots, x_n\rangle \subsetneq \mathbb{C}^{2n}.\]

The action of the \textbf{Weyl group} $W = \mathbb{Z}_2^n \rtimes S_n$ on $2^\Delta$ decomposes it into two classes, 
\[2^\Delta = \Omega_{n} \sqcup \overline{\Omega}_n\]
each defined by whether its elements contain the long root $n$ or not. Each class has these features:
\begin{itemize}
\item $I\in \overline{\Omega}_{n}=\{I:n\notin I\}$, $\Delta \setminus I = \{i_1, \ldots, i_s = n\}$
\begin{itemize}
    \item Equivalence classes in $\overline{\Omega}_n / \sim_W$ correspond to partitions $[k] = [1^{k_{1}}\cdots j^{k_{j}}\cdots n^{k_{n}}] \in \mathcal{P}_n$.
  \item The flag contains the lagrangian subspace $V_{i_s} = \langle x_1, \ldots, x_n\rangle=V_{i_s}^{\perp}$. 
    \item Levi subgroup $L_{[k]} = L_{I_{[k]}} = \prod_{j=1}^n \GL_{j}^{k_j}$ is  stabilizer of the splitting of the flag $V_{\bullet}\subset \mathbb{C}^{2n}$.
    \item Weyl group is $W_{[k]} = \mathbb{C}_2^{|[k]|} \rtimes S_{[k]}$ and normalizer is $N_{[k]} = L_{[k]} \rtimes W_{[k]}$, then strata are 
    \[\mathcal{X}_{\GL_{[k]}}^{\ast} /\!\!/ (\mathbb{Z}_2^{|[k]|} \rtimes S_{[k]})\; .\]
 \end{itemize}
\item $I\in \Omega_{n}=\{I:n\in I\}$ $\Delta \setminus I = \{i_1, \ldots, i_s<n\}$
\begin{itemize}
\item Equivalence classes in $\Omega_n / \sim_W$ correspond to partitions $I_{n, [k]}:=[k] \cup \{n\}$, $[k]= [1^{k_{1}}\cdots j^{k_{j}}\cdots m^{k_{m}}] \in \mathcal{P}_m$, $m < n$.
\item The flag does not contain the lagrangian subspace $\langle x_1, \ldots, x_n\rangle$.
\item  Levi subgroup $L_{n,[k]} = L_{I_{[k]}} = \prod_{j=1}^m \GL_{j}^{k_j} \times \Sp_{2(n-m)}$ is stabilizer of the splitting of the flag $V_{\bullet}\subset \mathbb{C}^{2n}$.
\item Weyl group is $W_{n,[k]} = \mathbb{Z}_2^{|[k]|} \rtimes S_{[k]}$ and normalizer is $N_{n,[k]} = L_{n,[k]} \rtimes W_{n,[k]}$, then strata are \[\mathcal{X}_{\GL_{[k]}\times \Sp_{2(n-m)}}^{\ast} /\!\!/ (\mathbb{Z}_2^{|[k]|} \rtimes S_{[k]})\; .\]
\end{itemize}
\end{itemize}

From this, we can decompose motivically the $\Sp_{2n}$-character variety. 
\begin{Theorem}\cite{gonzalez2024rootdata}\label{thm:stratSp}
The parabolic stratification of the $\Sp_{2n}$-character variety of a group $\Gamma$ yields the following decomposition of the $e$-polynomial:
\begin{align*} e\left( \mathcal{X}_{\Sp_{2n}}(\Gamma) \right) & = \sum_{[k]\in\mathcal{P}_{n}} e\left(\mathcal{X}_{\GL_{[k]}}^{\ast}(\Gamma) /\!\!/(\mathbb{Z}_2^{|[k]|} \rtimes S_{[k]}) \right) \\
& + \sum_{m=1}^{n-1} \sum_{[k] \in \mathcal{P}_m} e\left(\mathcal{X}_{\GL_{[k]}\times \Sp_{2(n-m)}}^{\ast}(\Gamma) /\!\!/ (\mathbb{Z}_2^{|[k]|} \rtimes S_{[k]}) \right)\; .\end{align*}
\end{Theorem}

Now we reproduce the analogue for the group $^L G=\SO_{2n+1}$ whose Dynkin diagram  $\dynkin[labels={1,2,n-2,n-1,n},label directions={,,,,,},scale=2]B{}$ of type $B_{n}$. Here, 
simple roots are $\Delta=\{1,\ldots,n\}$, which $n$ the unique short root. 

Fix a basis of $\mathbb{C}^{2n+1}$ with respect to which the quadratic form $Q:\mathbb{C}^{2n+1}\rightarrow \mathbb{C}$ is 
$Q(v)=x_1y_1+\cdots+x_ny_n+z^2$, where $(x_1,\ldots, x_n, y_1\ldots, y_n,z)$ are the coordinates of $v\in\mathbb{C}^{2n+1}$ in the chosen basis. Again, subsets $I \subseteq \Delta$, $\Delta \setminus I = \{i_1, \ldots, i_s\}$ are in bijection with standard parabolic subgroups $P_I$ which are stabilizers of the flag 
\[0 \subsetneq V_{1} \subsetneq \cdots \subsetneq V_{j} \subsetneq \cdots \subsetneq V_{s} \subsetneq V_{s}^\perp \subsetneq V_{{s-1}}^\perp \subsetneq \cdots \subsetneq V_{1}^\perp \subsetneq \mathbb{C}^{2n+1}\; ,\]  
where $V_{j} = \langle x_{1}, \ldots, x_{i_j}\rangle$ for $1 \leq j \leq s$ are \textbf{isotropic subspaces} and $V_j^\perp = \langle x_1, \ldots, x_n, y_{i_j+1}, \ldots, y_n, z\rangle$.

The action of the Weyl group $W = \mathbb{Z}_2^n \rtimes S_n$ on $2^\Delta$ decomposes it again into two classes,
\[2^\Delta =\Omega_{n} \sqcup \overline{\Omega}_n\]
each defined by whether its elements contain short root $n$ or not. Each class has these properties:
\begin{itemize}
    \item $I\in \overline{\Omega}_{n}=\{I:n\notin I\}$, $\Delta \setminus I = \{i_1, \ldots, i_s = n\}$
\begin{itemize}
    \item Equivalence classes in $\overline{\Omega}_n / \sim_W$ correspond to partitions $ [k] = [1^{k_{1}}\cdots j^{k_{j}}\cdots n^{k_{n}}] \in \mathcal{P}_n$.
  \item The flag contains the maximal isotropic subspace  $V_{i_s} = \langle x_1, \ldots, x_n\rangle$, with \linebreak $V_{i_s}^{\perp}=\langle x_1, \ldots, x_n, z\rangle$.
    \item Levi subgroup $L_{[k]} = L_{I_{[k]}} = \prod_{j=1}^n \GL_{j}^{k_j}$ is the stabilizer of the splitting (acting trivially on $\langle z\rangle = V_s^{\perp}/V_s$).
    \item Weyl group is $W_{[k]} = \mathbb{Z}_2^{|[k]|} \rtimes S_{[k]}$ and normalizer is $N_{[k]} = L_{[k]} \rtimes W_{[k]}$, then strata are \[\mathcal{X}_{\GL_{[k]}}^{\ast} /\!\!/ (\mathbb{Z}_2^{|[k]|} \rtimes S_{[k]})\; .\]
 \end{itemize} 
\item $I\in \Omega_{n}=\{I:n\in I\}$ $\Delta \setminus I = \{i_1, \ldots, i_s<n\}$
\begin{itemize}
\item Equivalence classes in $\Omega_n / \sim_W$ correspond to partitions $I_{n, [k]}:=[k] \cup \{n\}$, \linebreak $[k]= [1^{k_{1}}\cdots j^{k_{j}}\cdots m^{k_{m}}] \in \mathcal{P}_m$, $m < n$.
\item The flag does not contain the maximal isotropic subspace $\langle x_1, \ldots, x_n\rangle$.
\item  Levi subgroup $L_{n,[k]} = L_{I_{[k]}} = \prod_{j=1}^m \GL_{j}^{k_j} \times \SO_{2(n-m)+1}$ is the stabilizer of the splitting. 
\item Weyl group is $W_{n,[k]} = \mathbb{Z}_2^{|[k]|} \rtimes S_{[k]}$ and normalizer is $N_{n,[k]} = L_{n,[k]} \rtimes W_{n,[k]}$, then strata are \[\mathcal{X}_{\GL_{[k]}\times \SO_{2(n-m)+1}}^{\ast} /\!\!/ (\mathbb{Z}_2^{|[k]|} \rtimes S_{[k]})\; .\]
\end{itemize}
\end{itemize}

Again, this allows to decompose motivically the $\SO_{2n+1}$-character variety. 

\begin{Theorem}\cite{gonzalez2024rootdata}\label{thm:stratSOodd}
The parabolic stratification of the $\SO_{2n+1}$-character variety of a group $\Gamma$
yields the following decomposition of the $e$-polynomial:
\begin{align*} e\left( \mathcal{X}_{\SO_{2n+1}}(\Gamma) \right) & = \sum_{[k]\in\mathcal{P}_{n}} e\left(\mathcal{X}_{\GL_{[k]}}^{\ast}(\Gamma) /\!\!/(\mathbb{Z}_2^{|[k]|} \rtimes S_{[k]}) \right)\\
& + \sum_{m=1}^{n-1} \sum_{[k] \in \mathcal{P}_m} e\left(\mathcal{X}_{\GL_{[k]}\times \SO_{2(n-m)+1}}^{\ast}(\Gamma) /\!\!/ (\mathbb{Z}_2^{|[k]|} \rtimes S_{[k]}) \right).
\end{align*}
\end{Theorem}

By comparing the expressions in Theorems \ref{thm:stratSp} and \ref{thm:stratSOodd}, we observe that topological mirror symmetry holds for the Langlands dual groups $G=\Sp_{2n}$ and $^L G= \SO_{2n+1}$ if and only if it holds strata by strata. Notice that first summand is equal in both formulae, then equality holds if it does for the second summand. 

\begin{Corollary}\cite{gonzalez2024rootdata}\label{cor:TMSSpSO}
For the character varieties of the Langlands dual groups $\Sp_{2n}$ and $\SO_{2n+1}$, topological mirror symmetry (i.e. equality of motives or $e$-polynomials) reduces to show it for the irreducible character varieties $\mathcal{X}^{\ast}_{\Sp_{2n}}(\Gamma)$ and $\mathcal{X}^{\ast}_{\SO_{2n+1}}(\Gamma)$ and for each
\[\mathcal{X}_{\GL_{[k]} \times \Sp_{2(n-m)}}^{\ast}(\Gamma)/\!\!/ (\mathbb{Z}_2^{|[k]|} \rtimes S_{[k]})\quad \textrm{and}\quad \mathcal{X}_{\GL_{[k]}\times \SO_{2(n-m)+1}}^{\ast}(\Gamma)/\!\!/(\mathbb{Z}_2^{|[k]|} \rtimes S_{[k]})\; ,\]
where $[k]\in\mathcal{P}_m$ and $m=1,\ldots, n-1$. 
\end{Corollary}

Note that the decomposition in the stratification for $G=\Sp_{2n}$- and $^L G=\SO_{2n+1}$-character varieties involve strata for $\GL_r$ and groups of the same Dynkin type as $G$ and $^L G$ of lower or equal dimension. This reduces the computation in a recursive way to lower rank invariants. Nevertheless, observe that we still need to prove $e\left( \mathcal{X}^{\ast}_{\Sp_{2n}}(\Gamma) \right) =e\left( \mathcal{X}^{\ast}_{\SO_{2n+1}}(\Gamma) \right)$ for irreducible representations. 

To summarize the results in \cite{gonzalez2024rootdata} we can say that to compute the $e$-polynomial $e\left( \mathcal{X}^{\ast}_{G}(\Gamma) \right)$ we need to compute the $e$-polynomial of the irreducible locus $e\left( \mathcal{X}^{\ast}_{G}(\Gamma) \right)$ and, then, use the polystable stratification to reduce the computation of the lower strata to cores, yielding the $e$-polynomials $e\left( \mathcal{X}^{\ast}_{L_{I}}(\Gamma) /\!\!/W_I\right)$. There are two advantages in doing this: we gain recursion by leaving the computation in terms of previously known invariants, plus we substitute more  complicated infinite quotients in cohomology by finite ones. 

We end this subsection with an example particularizing the situation for the Langlands dual groups $\Sp_6$ and $\SO_7$.

\begin{Example}\label{ex:Sp6SO7}
The Dynkin diagram of $G=\Sp_{6}$ is $\dynkin[labels={1,2,3},scale=2]C3$ of type $C_{3}$ with simple roots labeled as $\Delta=\{1,2,3\}$, $n$ being the unique long root. The basis of $\mathbb{C}^6$ that we fix is $( x_1, x_2, x_3, y_1, y_2, y_3)$, with $\omega$ is the standard symplectic form 
$\left(\begin{array}{c|c}
    0 & I_3 \\\hline
    -I_3 & 0
\end{array}\right)$. The full flag in this case is 
\begin{eqnarray*}
    0 \subsetneq V_{1}=\langle x_1\rangle  \subsetneq V_{2}=\langle x_1, x_2\rangle\subsetneq V_{3}=\langle x_1, x_2, x_3\rangle= V_3^{\perp} \subsetneq \\
    V_{2}^\perp=\langle x_1, x_2, x_3, y_3\rangle \subsetneq V_{{1}}^\perp=\langle x_1, x_2, x_3, y_2, y_3\rangle \subsetneq \mathbb{C}^{6}\; . 
\end{eqnarray*}

On the other hand, the Dynkin diagram of $G=\SO_{7}$ is $\dynkin[labels={1,2,3},scale=2]B3$ of type $B_{3}$ with simple roots labeled as $\Delta=\{1,2,3\}$, $n$ being the unique short root. The basis of $\mathbb{C}^7$ that now we fix is $( x_1, x_2, x_3, z, y_1, y_2, y_3)$, with $\omega$ being the quadratic form of expression $x_1y_1+x_2y_2+x_3y_3+z^2$, in coordinates of the basis.  
The full flag in this case is 
\begin{eqnarray*}
    0 \subsetneq V_{1}=\langle x_1\rangle  \subsetneq V_{2}=\langle x_1, x_2\rangle\subsetneq V_{3}=\langle x_1, x_2, x_3\rangle\subsetneq V_3^{\perp}=\langle x_1, x_2, x_3, z\rangle\subsetneq \\
    V_{2}^\perp=\langle x_1, x_2, x_3, z, y_3\rangle \subsetneq V_{{1}}^\perp=\langle x_1, x_2, x_3, z, y_2, y_3\rangle \subsetneq \mathbb{C}^{7}\; .
\end{eqnarray*}

For both cases, given a subset $I \subseteq \Delta$, we remove the terms indexed by elements of $I$ in each full flag and obtain that $P_I$ is conjugated to the stabilizer of that reduced flag, and also that $L_I$ is conjugated to the stabilizer of the graded object of the flag. Let us see this for two particular subsets $I$.

First, suppose $I=\{2\}$. For $G=\Sp_6$ we obtain: 
\[P_I\;\text{stabilizes}\; \quad 0\subsetneq V_1\subsetneq V_3=V_3^{\perp} \subsetneq V_1^{\perp}\subsetneq \mathbb{C}^6\; ,\]
\[L_I\;\text{stabilizes}\;\quad V_1\oplus V_3/V_1\oplus V_1^{\perp}/V_3^{\perp}\oplus \mathbb{C}^6/V_1^{\perp}\; ,\]
where note that $L_I$ contains the action of $\GL_1$ on the first factor and $\GL_2$ on the second factor, and the action on the other two factors is the induced one because of the symplectic form, then $L_I\simeq \GL_1\times \GL_2$. 

Now for $^L G=\SO_7$ we obtain:
\[P_I\;\text{stabilizes}\; \quad 0\subsetneq V_1\subsetneq V_3\subsetneq V_3^{\perp} \subsetneq V_1^{\perp}\subsetneq \mathbb{C}^7\; ,\]
\[L_I\;\text{stabilizes}\;\quad V_1\oplus V_3/V_1\oplus  V_3^{\perp}/V_3\oplus V_1^{\perp}/V_3^{\perp}\oplus \mathbb{C}^7/V_1^{\perp}\; .\]
Now, note that $L_I$ contains again the action of $\GL_1$ on the first factor and $\GL_2$ on the second factor. On the third factor $V_3^{\perp}/V_3$ the action is trivial, and on the remaining two factors the action is the induced one because of the orthogonal form, then $L_I\simeq \GL_1\times \GL_2$. 

Then, using Theorems \ref{thm:moticivGZ}, \ref{thm:stratSp} and \ref{thm:stratSOodd}, we can compute the $e$-polynomial of the strata as 
\[e\left(\mathcal{X}^{I=\{2\}}_{\Sp_6}(\Gamma)\right)=e\left(\mathcal{X}_{\GL_1}^{\ast}\times\mathcal{X}_{\GL_2}^{\ast}/\mathbb{Z}_2^2\right)=e\left(\mathcal{X}^{I=\{2\}}_{\SO_7}(\Gamma)\right)\]
noticing that they are automatically equal because both the Levi $L_I$ and the finite group in the quotient are isomorphic. 

Now suppose $I=\{2,3\}$. For $G=\Sp_6$ we obtain: 
\[P_I\;\text{stabilizes}\; \quad 0\subsetneq V_1\subsetneq V_1^{\perp}\subsetneq \mathbb{C}^6\; ,\]
\[L_I\;\text{stabilizes}\;\quad V_1\oplus V_1^{\perp}/V_q^{\perp}\oplus \mathbb{C}^6/V_1^{\perp}\; .\]
Here, $L_I$ corresponds to the action of $\GL_1$ on the first factor (inducing the same on the third factor) and $\Sp_4$ acts on the second factor, then $L_I\simeq \GL_1\times \Sp_4$. 

For $^L G=\SO_7$ we obtain:
\[P_I\;\text{stabilizes}\; \quad 0\subsetneq V_1\subsetneq V_1^{\perp}\subsetneq\mathbb{C}^7\; ,\]
\[L_I\;\text{stabilizes}\;\quad V_1\oplus  V_1^{\perp}/V_1\oplus \mathbb{C}^7/V_1^{\perp}\; .\]
Now $L_I$ contains the action of $\GL_1$ on the first factor (inducing the same on the third factor) and $\SO_5$ acts on the second factor, then $L_I\simeq \GL_1\times \SO_5$. 

Once more, using Theorems \ref{thm:moticivGZ}, \ref{thm:stratSp} and \ref{thm:stratSOodd}, we can compute the $e$-polynomial of the strata as 
\begin{equation}\label{eq:TMSSpSO}
e\left(\mathcal{X}^{I=\{2,3\}}_{\Sp_6}(\Gamma)\right)=e\left(\mathcal{X}_{\GL_1}^{\ast}\times\mathcal{X}_{\Sp_4}^{\ast}/\mathbb{Z}_2\right)\overset{?}{=}e\left(\mathcal{X}_{\GL_1}^{\ast}\times\mathcal{X}_{\SO_5}^{\ast}/\mathbb{Z}_2\right)=e\left(\mathcal{X}^{I=\{2,3\}}_{\SO_7}(\Gamma)\right).
\end{equation}
Now, however, we cannot assure that these polynomials are equal because the Levi subgroups $\GL_1\times \Sp_4$ and $\GL_1\times \SO_5$ are not isomorphic: $\Sp_4(\mathbb{C})\simeq \Spin_5(\mathbb{C})$ is the double cover of $\SO_5(\mathbb{C})$. 

In the spirit of Corollary \ref{cor:TMSSpSO} it can be checked that showing the topological mirror symmetry conjecture in the form of the equality of $e$-polynomials $e\left(\mathcal{X}_{\Sp_6}(\Gamma)\right)=e\left(\mathcal{X}_{\SO_7}(\Gamma)\right)$ reduces to show the equality (\ref{eq:TMSSpSO}) for the strata $I=\{2,3\}$, the equality for the strata $I=\{1,3\}$ and the equality $e\left(\mathcal{X}^{\ast}_{\Sp_6}(\Gamma)\right)=e\left(\mathcal{X}^{\ast}
_{\SO_7}(\Gamma)\right)$ for the irreducible representations. The only thing to prove for strata $\{1,3\}$ is the equality for character varieties of $\Sp_2\simeq \SL_2$ and $\SO_3\simeq \PGL_2$, which is already proven for $\Gamma$ a surface group \cite{hausel2008mixed} and a free group \cite{florentino2021serre}, the general $\Gamma$ case being open. 
    
\end{Example}

\subsection{Stratification for \texorpdfstring{$\SO_{2n}$}{SO(2n)}, type D}

We conclude by exploring the stratification for $\SO_{2n}$ of Dynkin diagram  
\[\dynkin[labels={1,2,n-3,n-2,n-1,n},label directions={,,,right,,},scale=2]D{}\] and type $D_{n}$. Simple roots are $\Delta=\{1,\ldots,n\}$, with $n-1,n$ being the end roots of the fork. This group is Langlands self-dual, therefore the discussion about topological mirror symmetry does not apply in this case. 

As in the odd orthogonal case, we fix a basis of $\mathbb{C}^{2n}$ such that the quadratic form $Q:\mathbb{C}^{2n}\rightarrow \mathbb{C}$ is 
$Q(v)=x_1y_1+\cdots+x_ny_n$, $(x_1,\ldots, x_n, y_1\ldots, y_n)$ being the coordinates of $v\in\mathbb{C}^{2n}$ in the basis. In this case there is no bijection between parabolic subgroups of $\SO_{2n}$ and isotropic flags, then we need to analyze it more carefully. 

The action of the Weyl group $W =  H_n\rtimes S_n$ on $2^\Delta$, where 
\[H_n=\ker(\epsilon_1,\ldots,\epsilon_n\in\mathbb{Z}_2^n\mapsto\epsilon_1\cdot \cdots \cdot \epsilon_n\in\mathbb{Z}_2)\; ,\]
decomposes 
\[2^\Delta = \Omega^1_n \sqcup \Omega^2_n\sqcup \Omega^3_n\]
into three classes, depending on containing the end roots of the fork or not. Let us discuss the properties for each class: 
\begin{itemize}
\item $I\in \Omega^1_{n}=\{I: \{n-1, n\} \subseteq I\}$, $\Delta \setminus I = \{i_1, \ldots, i_s \le n-2\}$
\begin{itemize}
\item Equivalence classes in $\Omega^1_n / \sim_W$ correspond to partitions $I_{1, [k]}:=[k] \cup \linebreak \{n-d+1, \ldots, n\}$, $[k]= [1^{k_{1}}\cdots j^{k_{j}}\cdots m^{k_{m}}] \in \mathcal{P}_m$, $0 \le m \le n-2$ (where $d = n-m \ge 2$).
\item $P_I$ is the stabilizer of the flag 
\[0 \subsetneq V_{1} \subsetneq \cdots \subsetneq V_{j} \subsetneq \cdots \subsetneq V_{s-1} \subsetneq V_{s} \subsetneq V_{s}^\perp \subsetneq V_{{s-1}}^\perp \subsetneq \cdots \subsetneq V_{1}^\perp \subsetneq \mathbb{C}^{2n}\; ,\]
where $V_j = \langle x_1, \ldots, x_{i_j}\rangle$ and  $V_j^\perp = \langle x_1, \ldots, x_n, y_{i_j+1}, \ldots, y_n\rangle$. 
    \item Levi subgroup is $
    L_{1,[k]} = \prod_{j=1}^m \GL_{j}^{k_j} \times \SO_{2(n-m)}$ and strata are 
    \[\mathcal{X}_{\GL_{[k]} \times \SO_{2(n-m)}}^{\ast} /\!\!/ (\mathbb{Z}_2^{|[k]|} \rtimes S_{[k]})\; .\]
    When $m=0$ ($I=\Delta$), this recovers the open stratum $\mathcal{X}_{\SO_{2n}}^*$.
 \end{itemize}
\item $I\in \Omega^2_{n}=\{I:n\not\in I, n-1\in I\}$, $\Delta \setminus I = \{i_1, \ldots, i_{s-1}<n-1, i_s = n\}$
\begin{itemize}
\item Equivalence classes in $\Omega^2_n / \sim_W$ correspond to partitions $I_{2, [k]}:=[k] \cup \{n-1\}$, $[k]= [1^{k_{1}}\cdots j^{k_{j}}\cdots n^{k_{n}}] \in \mathcal{P}_n$.
\item $P_I$ is the stabilizer of the flag 
\[0 \subsetneq V_{1} \subsetneq \cdots \subsetneq V_{j} \subsetneq \cdots \subsetneq V_{s-1} \subsetneq V_{s} =V_{s}^\perp \subsetneq V_{{s-1}}^\perp \subsetneq \cdots \subsetneq V_{1}^\perp \subsetneq \mathbb{C}^{2n}\; ,\]
where $V_j = \langle x_1, \ldots, x_{i_j}\rangle$ and $V_s = \langle x_1, \ldots, x_n\rangle$.
    \item Levi subgroup is $
    L_{2,[k]} = \prod_{j=1}^n \GL_{j}^{k_j}$ and strata are 
   \[\mathcal{X}_{\GL_{[k]}}^{\ast} /\!\!/ (H_{[k]}\rtimes S_{[k]})\; .\]
 \end{itemize}
\item $I\in \Omega^3_{n}=\{I:n-1\not\in I, n\in I \text{ or } \{n-1,n\}\cap I = \emptyset\}$, $\Delta \setminus I$ omitting at least $n-1$.
\begin{itemize}
\item Equivalence classes in $\Omega^3_n / \sim_W$ correspond to partitions $I_{3, [k]}:=[k] \cup \{n\}$ (or $[k]$), $[k] \in \mathcal{P}_n$.
\item $P_I$ is the stabilizer of the flag 
\[0 \subsetneq V_{1} \subsetneq \cdots \subsetneq V_{j} \subsetneq \cdots \subsetneq V_{s-1} \subsetneq V_{s}' =V_{s}^{'\perp} \subsetneq V_{{s-1}}^\perp \subsetneq \cdots \subsetneq V_{1}^\perp \subsetneq \mathbb{C}^{2n}\; ,\]
where $V_s' = \langle x_1, \ldots, x_{n-1}, y_n\rangle$.
    \item Levi subgroup is $
    L_{3,[k]} = \prod_{j=1}^n \GL_{j}^{k_j}$ and strata are 
   \[\mathcal{X}_{\GL_{[k]}}^{\ast} /\!\!/ (H_{[k]}\rtimes S_{[k]})\; .\]
 \end{itemize}
\end{itemize}

Whether $I_{2,[k]}$ and $I_{3,[k]}$ belong to the same $W$-orbit depends on the parity of $[k]$: if $[k] \in \mathcal{P}_n^{\text{even}}$ (all parts even), swapping $n-1$ and $n$ requires an outer transformation in $\mathrm{O}_{2n} \setminus \SO_{2n}$, giving two distinct orbits under $W$; if $[k] \in \mathcal{P}_n^{\text{odd}}$ (some part odd), the two classes merge under $W$. Note that if $n$ is odd, $\mathcal{P}_n^{\text{even}} = \emptyset$.

As before, we can decompose motivically the $\SO_{2n}$-character variety.

\begin{Theorem}\cite{gonzalez2024rootdata}\label{thm:stratSOeven}
The parabolic stratification of the $\SO_{2n}$-character variety of a group $\Gamma$
yields the following decomposition of the $e$-polynomial:
\[e\left( \mathcal{X}_{\SO_{2n}}(\Gamma) \right) = \sum_{m = 0}^{n-2}\sum_{[k]\in\mathcal{P}_{m}} e\left(\mathcal{X}_{\GL_{[k]}\times \SO_{2(n-m)}}^{\ast}(\Gamma)/\!\!/ (\mathbb{Z}_2^{|[k]|} \rtimes S_{[k]}) \right) + \]
\[\sum_{[k]\in\mathcal{P}_{n}^{\textup{odd}}} e\left(\mathcal{X}_{\GL_{[k]}}^{\ast}(\Gamma) /\!\!/ (H_{[k]} \rtimes S_{[k]}) \right)
+ 2 \sum_{[k] \in \mathcal{P}_n^{\textup{even}}} e\left(\mathcal{X}_{\GL_{[k]}}^{\ast}(\Gamma) /\!\!/ (H_{[k]} \rtimes S_{[k]}) \right).\]
\end{Theorem}

%\bibliographystyle{alpha} %
%\bibliography{referencesAZ}
\printbibliography
\end{document}